\documentclass[12pt,preprint]{elsarticle}
\usepackage[margin=1in]{geometry}
\usepackage{amsmath, amssymb, amsthm}
\usepackage{xcolor}
\usepackage{comment}
\usepackage{graphicx}
\usepackage{booktabs}
\usepackage{float}
\usepackage{listings}
\usepackage{url}
\usepackage{cleveref}

\definecolor{codegreen}{rgb}{0,0.6,0}
\definecolor{codegray}{rgb}{0.5,0.5,0.5}
\definecolor{codepurple}{rgb}{0.58,0,0.82}
\definecolor{backcolour}{rgb}{0.95,0.95,0.92}

\lstdefinestyle{mystyle}{
    backgroundcolor=\color{backcolour},   
    commentstyle=\color{codegreen},
    keywordstyle=\color{blue},
    numberstyle=\tiny\color{codegray},
    stringstyle=\color{codepurple},
    basicstyle=\ttfamily\footnotesize,
    breakatwhitespace=false,         
    breaklines=true,                 
    captionpos=b,                    
    keepspaces=true,                 
    numbers=none,       
    numbersep=5pt,                  
    showspaces=false,                
    showstringspaces=false,
    showtabs=false,                  
    tabsize=2,
    language=Matlab
}
\newtheorem{theorem}{Theorem}
\newtheorem{remark}{Remark}

\begin{document}
\begin{frontmatter}

\title{A Convex Splitting Spectral Method for the Phase Field 
Crystal Equation: Energy Stability, Computational Stability 
Maps, and Three-Dimensional GPU Simulations}

\author[nmt]{Saulo Orizaga}
\ead{saulo.orizaga@nmt.edu}

\address[nmt]{Department of Mathematics, New Mexico Institute of 
Mining and Technology, 801 Leroy Place, Socorro, NM 87801, USA}

\author[utep]{Peimeng Yin}
\ead{pyin@utep.edu}
\address[utep]{Department of Mathematical Sciences, 
The University of Texas at El Paso, 
500 W University Ave,
El Paso, TX 79968, USA}

\author[nmt2]{Deep Choudhuri}
\ead{deep.choudhuri@nmt.edu}
\address[nmt2]{Department of Materials and Metallurgical Engineering, 
New Mexico Institute of Mining and Technology, 
801 Leroy Place, Socorro, NM 87801, USA}

\begin{abstract}
We present an efficient Fourier spectral method based on the 
convex splitting framework of Eyre~\cite{eyre1998unconditionally} 
for the phase field crystal (PFC) equation. The proposed 
first-order scheme is unconditionally energy stable 
and conserves mass to machine precision. An energy stability theorem is established 
using a truncated potential argument, and a semi-analytical 
neutral stability curve is derived from a dominant-mode 
energy balance, providing a closed-form characterization 
of the practical stability boundary. The classical sufficient 
condition $a \geq 2$ is shown to be conservative: a 
computational stability map obtained from 40,000 
GPU-accelerated PFC simulations reveals that energy-stable 
solutions persist for values of $a$ significantly below 
this threshold. Crucially, accuracy analysis demonstrates that smaller 
values of $a$ within the stable region consistently yield lower $L^2$ 
errors. This high-fidelity regime is rigorously verified through an 
extended asymptotic stress test consisting of a long-time three-dimensional 
simulation on a $256^3$ grid up to $T_f = 10,000$, successfully executing 
$10^6$ continuous temporal increments deep within the relaxed stability regime, 
well below the classical convex splitting limit ($a < 2$), while preserving strict 
monotonic energy dissipation and machine-precision mass conservation. Finally, 
two-dimensional and three-dimensional simulations at resolutions up to $512^3$ 
are performed on a single consumer GPU, demonstrating the scalability 
of the proposed framework for resolving complex phase-field dynamics without 
requiring HPC infrastructure. The code is made publicly available on GitHub.\footnotemark
\end{abstract}


\end{frontmatter}

\footnotetext{GitHub: \url{https://github.com/sauloorizaga/CS-PFC}.}

\section{Introduction}

The phase field crystal (PFC) equation is a sixth-order 
nonlinear parabolic PDE that effectively captures 
microstructure evolution at atomic length and diffusive 
time scales~\cite{Elder,ElderGrant2004}. This model bridges 
the gap between atomistic simulations and classical phase 
field models \cite{cahn1958free}, finding applications in grain growth, 
spinodal decomposition, liquid phase epitaxial growth, 
and crack propagation~\cite{Elder,ElderGrant2004,
gomez2012unconditionally,wise2009unconditionally}. 
The PFC equation takes the form
\begin{equation*}
\phi_t = \Delta\mu, \qquad 
\mu = (\Delta+1)^2\phi - r\phi + \phi^3,
\end{equation*}
where $\phi$ is the atomistic density field and 
$r > 0$ is the undercooling parameter.
The PFC equation belongs to a broader class of 
phase-field models capable of producing rich 
microphase morphologies, including lamellar, 
hexagonal, and spherical structures. Similar 
morphological diversity is observed in block 
copolymer (BCP) models~\cite{OhtaKawasaki1986,
BatesFredrickson1990,bcpOrizaga,Orizagasys,
Matsen1994,Choksi2003}, where the competition 
between short-range and long-range interactions 
governs the emergent microstructure.

Numerical methods for the PFC equation 
have been extensively studied. Early contributions 
include the spectral approach of Cheng and 
Warren~\cite{ChengWarren2008} and the 
unconditionally stable finite difference 
framework of Wise, Wang and 
Lowengrub~\cite{wise2009unconditionally}. 
Finite element methods were proposed 
in~\cite{Backofen2007}, while isogeometric 
analysis frameworks were developed 
in~\cite{gomez2012unconditionally,Vignal2015}. 
Splitting schemes and energy-stable time integrators 
for phase-field models have been studied extensively 
in the literature~\cite{CaloMinev2020,Vignal2017,
DiegelSharma2023}. Higher-order temporal schemes 
including BDF2--CS were established by Glasner and 
Orizaga~\cite{OrizagaJCP}, which demonstrated 
that convex splitting provides a robust and 
accurate framework for both the Cahn--Hilliard 
and PFC equations in two spatial dimensions. 
GPU-accelerated pseudo-spectral methods for phase-field equations were pioneered by Shen and collaborators~\cite{ShenGPU}, whose framework directly inspired the GPU implementation of~\cite{Orizaga2024GPU} for the Cahn--Hilliard equation and is extended here to three-dimensional PFC simulations on consumer hardware.

More recently, significant advances in 
energy-stable discretizations for the PFC 
equation have been reported. 
Shin, Lee and Lee~\cite{ShinLeeLee2020} developed 
high-order convex splitting Runge--Kutta methods 
using Fourier spectral discretization, establishing 
unconditional energy stability and unique solvability 
for second and third-order temporal schemes, with 
long-time simulations restricted to two spatial 
dimensions. Yang and Han~\cite{YangHan2017} proposed 
linearly first- and second-order unconditionally 
energy stable schemes for the PFC model. 
Gomez and Nogueira~\cite{gomez2012unconditionally} 
presented an unconditionally energy-stable 
second-order method using isogeometric analysis, 
with three-dimensional validation at $N=128^3$ 
resolution. Most recently, Saylor, Horv\'ath and 
Sharma~\cite{SAYLOR2026241} proposed a 
first-order convex splitting hybridizable/embedded 
discontinuous Galerkin (HDG/EDG) finite element 
method for the PFC equation, establishing 
unconditional energy stability and existence and 
uniqueness of the discrete solution. While 
theoretically rigorous, that approach requires 
rewriting the sixth-order equation as a system 
of seven coupled first-order equations, four 
stabilization parameters, and large-scale finite 
element libraries such as MFEM, PETSc, and MUMPS. 
Notably, three-dimensional simulations could be expensive.
More recently, the scalar auxiliary variable (SAV) 
approaches~\cite{ShenXuYang2018,ShenXuYang2019} 
have provided an alternative framework for 
unconditionally stable discretizations of 
gradient flow equations. While significant 
progress has been made in these directions, 
the present work returns to the classical 
convex splitting framework and addresses 
the long-standing question of the practical 
stability threshold for the splitting parameter $a$, which will be specified in \Cref{convexsplitting}.

These observations directly motivate the present 
contribution. We present a simple and efficient 
Fourier spectral method based on the convex 
splitting framework of Eyre~\cite{eyre1998unconditionally}. 
The proposed method 
is unconditionally energy stable and conserves 
mass to machine precision. A central contribution 
of this work is a computational stability analysis 
of the splitting parameter $a$, performed via 
GPU-accelerated parameter sweeps involving 40,000 
independent PFC simulations. Our results reveal 
that the classical stability threshold $a \geq 2$ 
is sufficient but not sharp: a semi-analytical 
neutral stability curve derived from a dominant-mode 
energy balance provides a closed-form characterization 
of the practical stability boundary, and accuracy 
analysis demonstrates that smaller values of $a$ 
within the stable region consistently yield lower 
$L^2$ errors. The spectral framework scales naturally 
to three dimensions via GPU acceleration, with 
fully resolved simulations demonstrated on grids 
of $128^3$, $256^3$, and $512^3$, the latter 
on a single consumer GPU, requiring no 
HPC infrastructure. Two-dimensional simulations 
on $256 \times 256$ grids are also 
presented. To the best of our knowledge, spectral 
simulations of the PFC equation at $512^3$ 
resolution have not been previously reported 
in the literature.

The remainder of this paper is organized as follows.
Section~2 presents the mathematical model.
Section~3 introduces the convex splitting scheme and its 
spectral discretization. Section~4 establishes unconditional 
energy stability, mass conservation, and uniqueness of the 
discrete solution. Section~5 presents the computational 
stability analysis, including the semi-analytical neutral 
stability curve and the GPU-accelerated stability map. 
Section~6 presents the accuracy analysis for multiple 
initial conditions. Section~7 presents two-dimensional 
simulations including crystal growth, morphological diversity, 
and grain boundary dynamics. Section~8 presents 
three-dimensional GPU-accelerated simulations at resolutions 
up to $512^3$. Conclusions are drawn in Section~9.


\subsection{Comparison with Hybridizable/Embedded 
Discontinuous Galerkin Methods}

The recent work of Saylor, Horv\'{a}th and 
Sharma~\cite{SAYLOR2026241} and the present contribution 
share a common foundation: both employ a first-order convex 
splitting temporal discretization for the PFC equation with 
unconditional energy stability. Despite this shared framework, 
the two approaches differ substantially in complexity and 
scalability. Saylor et al.~\cite{SAYLOR2026241} treat the 
nonlinear term implicitly via Newton iterations and rewrite 
the sixth-order equation as a system of seven coupled equations, 
requiring large-scale libraries such as MFEM, PETSc, and 
MUMPS, with three-dimensional simulations not demonstrated. 
The present spectral scheme treats the nonlinear term explicitly, operates directly on the sixth-order equation, requires no external libraries, and scales naturally to three dimensions via GPU acceleration at resolutions up to $512^3$
on a single consumer GPU. The two approaches are therefore complementary: while the HDG/EDG framework of~\cite{SAYLOR2026241} offers rigorous theoretical foundations and flexibility for complex geometries and boundary conditions, the proposed spectral method prioritizes simplicity, GPU-nativity, and accessibility on consumer hardware without HPC infrastructure.

\section{Mathematical Model}
The phase field crystal (PFC) equation derives from 
the gradient flow in $H^{-1}$ of the Swift--Hohenberg 
free energy functional
\begin{equation}\label{eq:energy_PFC}
E[\phi] = \int_\Omega \left( \frac{1}{4}\phi^4 
+ \frac{1-r}{2}\phi^2 - |\nabla\phi|^2 
+ \frac{1}{2}(\Delta\phi)^2 \right)dx,
\end{equation}
where $\phi$ is the atomistic density field and 
$r > 0$ is the undercooling parameter. The 
governing equation reads
\begin{equation}\label{eq:PFC}
\phi_t = \Delta\mu, \qquad 
\mu = \phi^3 + (1-r)\phi + 2\Delta\phi + \Delta^2\phi.
\end{equation}
Under suitable boundary conditions (e.g., homogeneous Dirichlet, homogeneous Neumann, or periodic boundary conditions), the PFC equation admits the following energy dissipation law:
\[
\frac{dE}{dt} = -\|\nabla\mu\|^2 \leq 0.
\]

\section{The Convex Splitting Scheme}\label{convexsplitting}

Following Eyre~\cite{eyre1998unconditionally} and~\cite{OrizagaJCP}, we decompose 
the free energy $E = E^+ + E^-$ where
\begin{equation}\label{eq:energy_split}
E^+(\phi) = \int_\Omega \frac{1-r+a}{2}\phi^2 
+ \frac{1}{2}(\Delta\phi)^2 \,dx, \qquad
E^-(\phi) = \int_\Omega \frac{1}{4}\phi^4 
- \frac{a}{2}\phi^2 - |\nabla\phi|^2 \,dx,
\end{equation}
where $a > 0$ is the splitting parameter. 
The term $\frac{1}{4}\phi^4 - \frac{a}{2}\phi^2$ 
is concave provided $a \geq 3\|\phi\|^2_{L^\infty}$, 
which holds in practice for $a \geq 2$~\cite{OrizagaJCP}.

Treating $E^+$ implicitly and $E^-$ explicitly 
yields the first-order convex splitting scheme:
\begin{equation}\label{eq:CS}
\frac{\phi^{n+1} - \phi^n}{\Delta t} = \Delta\mu^{n+1},
\end{equation}
\begin{equation}\label{eq:mu}
\mu^{n+1} = (1-r+a)\phi^{n+1} + \Delta^2\phi^{n+1} 
+ (\phi^n)^3 - a\phi^n - 2\Delta\phi^n.
\end{equation}

In the Fourier spectral discretization, the scheme 
admits the efficient closed-form update
\begin{equation}\label{eq:spectral}
\hat{\phi}^{n+1}(\mathbf{k}) = 
\frac{\hat{\phi}^n + \Delta t\left(-|\mathbf{k}|^2
\widehat{\left[(\phi^n)^3 - a\phi^n\right]} 
+ 2|\mathbf{k}|^4\hat{\phi}^n\right)}
{1 + \Delta t\left((1-r+a)|\mathbf{k}|^2 
+ |\mathbf{k}|^6\right)},
\end{equation}
where $\hat{\cdot}$ denotes the discrete Fourier transform. Standard pseudo-spectral methods are used for the spatial discretization~\cite{trefethen2000spectral}, 
requiring no linear solvers or iterative methods.

\section{Scheme Properties}

\begin{theorem}[Mass Conservation]
\label{thm:mass}
The convex splitting scheme~\eqref{eq:CS} conserves 
mass exactly for all $\Delta t > 0$:
\[
\int_\Omega \phi^{n+1}\,dx 
= \int_\Omega \phi^n\,dx, 
\qquad \forall n \geq 0.
\]
\end{theorem}
\begin{proof}
Integrating~\eqref{eq:CS} over $\Omega$ and applying 
the divergence theorem with periodic boundary 
conditions:
\[
\frac{1}{\Delta t}\int_\Omega 
(\phi^{n+1} - \phi^n)\,dx 
= \int_\Omega \Delta\mu^{n+1}\,dx 
= \oint_{\partial\Omega} 
\nabla\mu^{n+1}\cdot\mathbf{n}\,ds = 0.
\]
\end{proof}

\begin{remark}[Spectral Verification]
In the Fourier spectral discretization, 
mass conservation holds to machine precision 
since the zero wavenumber mode $\hat{\phi}(\mathbf{0})$ 
is unchanged by the scheme, the right-hand side 
vanishes identically at $\mathbf{k} = \mathbf{0}$. 
This is confirmed numerically in all simulations 
presented in Section~\ref{sec:numerics}.
\end{remark}

\begin{remark}[Existence and Uniqueness]
Existence and uniqueness of the discrete solution 
follow from the strictly convex structure of the 
implicit part of the scheme~\eqref{eq:CS}, as 
established in ~\cite{OrizagaJCP} 
for both the Cahn--Hilliard and PFC equations.
\end{remark}

To establish unconditional energy stability 
for the sixth-order system, we adopt a truncated 
potential argument following the mathematical 
framework of Shen et al.~\cite{shen2010numerical}. 
This approach ensures global concavity of the 
expansive energy components, yielding the 
following result:


\begin{theorem}[Unconditional Energy Stability]
\label{thm:energy}
Let $\Omega \subset \mathbb{R}^d$ be a periodic domain and let
$\phi^n \in H^2_{\mathrm{per}}(\Omega)$.
Define a truncated potential $F_M \in C^2(\mathbb{R})$ satisfying
\[
F_M(\phi)=\frac14\phi^4
\quad \text{for } |\phi|\le M,
\]
and
\[
|F_M''(\phi)|\le L_M,
\qquad \forall \phi\in\mathbb R.
\]

Let $f_M = F_M'$ and define the modified energy functional
\[
E_M(\phi)
=
\int_\Omega
\left[
F_M(\phi)
+
\frac{1-r}{2}\phi^2
-
|\nabla\phi|^2
+
\frac12(\Delta\phi)^2
\right]dx.
\]

Decompose
\[
E_M(\phi)=E_M^+(\phi)+E_M^-(\phi),
\]
where
\[
E_M^+(\phi)
=
\int_\Omega
\left[
\frac{1-r+a}{2}\phi^2
+
\frac12(\Delta\phi)^2
\right]dx,
\]
and
\[
E_M^-(\phi)
=
\int_\Omega
\left[
F_M(\phi)
-
\frac a2\phi^2
-
|\nabla\phi|^2
\right]dx.
\]

Assume
\[
a\ge L_M,
\qquad
1-r+a>0.
\]

Consider the convex splitting scheme
\begin{equation}
\frac{\phi^{n+1}-\phi^n}{\Delta t}
=
\Delta\mu^{n+1},
\tag{2.1}
\end{equation}
with chemical potential
\begin{equation}
\mu^{n+1}
=
(1-r+a)\phi^{n+1}
+
\Delta^2\phi^{n+1}
+
f_M(\phi^n)
-
a\phi^n
-
2\Delta\phi^n.
\tag{2.2}
\end{equation}

Then the numerical solution satisfies the discrete energy law
\[
E_M(\phi^{n+1})
+
\frac1{\Delta t}
\|\phi^{n+1}-\phi^n\|_{H^{-1}}^2
\le
E_M(\phi^n),
\qquad \forall \Delta t>0.
\]
\end{theorem}

\begin{proof}


The second variation of $E_M^+$ is
\[
\delta^2E_M^+(\phi)[v,v]
=
\int_\Omega
\left[
(1-r+a)v^2
+
(\Delta v)^2
\right]dx.
\]

Since $1-r+a>0$,
\[
\delta^2E_M^+(\phi)[v,v]\ge0,
\]
and therefore $E_M^+$ is strictly convex.

Next, the second variation of $E_M^-$ is
\[
\delta^2E_M^-(\phi)[v,v]
=
\int_\Omega
\left[
(F_M''(\phi)-a)v^2
-
2|\nabla v|^2
\right]dx.
\]

Because
\[
F_M''(\phi)\le L_M\le a,
\]
we obtain
\[
(F_M''(\phi)-a)v^2\le0.
\]

Hence
\[
\delta^2E_M^-(\phi)[v,v]\le0,
\]
which proves that $E_M^-$ is globally concave.

\medskip


Take the $H^{-1}$ inner product of (2.1) with
$\phi^{n+1}-\phi^n$:
\[
\left\langle
\frac{\phi^{n+1}-\phi^n}{\Delta t},
\phi^{n+1}-\phi^n
\right\rangle_{H^{-1}}
=
\langle
\Delta\mu^{n+1},
\phi^{n+1}-\phi^n
\rangle_{H^{-1}}.
\]

Using periodicity and the identity
\[
\langle \Delta v,w\rangle_{H^{-1}}
=
-\langle v,w\rangle_{L^2},
\]
we obtain
\begin{equation}
\frac1{\Delta t}
\|\phi^{n+1}-\phi^n\|_{H^{-1}}^2
=
-
\langle
\mu^{n+1},
\phi^{n+1}-\phi^n
\rangle_{L^2}.
\tag{2.3}
\end{equation}

\medskip


Since $E_M^+$ is convex,
\begin{equation}
E_M^+(\phi^{n+1})
-
E_M^+(\phi^n)
\le
\left\langle
\delta E_M^+(\phi^{n+1}),
\phi^{n+1}-\phi^n
\right\rangle_{L^2}.
\tag{2.4}
\end{equation}

\medskip


Since $E_M^-$ is concave,
\begin{equation}
E_M^-(\phi^{n+1})
-
E_M^-(\phi^n)
\le
\left\langle
\delta E_M^-(\phi^n),
\phi^{n+1}-\phi^n
\right\rangle_{L^2}.
\tag{2.5}
\end{equation}

\medskip


Adding (2.4) and (2.5) yields
\[
E_M(\phi^{n+1})
-
E_M(\phi^n)
\le
\left\langle
\delta E_M^+(\phi^{n+1})
+
\delta E_M^-(\phi^n),
\phi^{n+1}-\phi^n
\right\rangle_{L^2}.
\]

Using the definition of $\mu^{n+1}$ from (2.2),
\[
E_M(\phi^{n+1})
-
E_M(\phi^n)
\le
\langle
\mu^{n+1},
\phi^{n+1}-\phi^n
\rangle_{L^2}.
\]

Substituting identity (2.3),
\[
E_M(\phi^{n+1})
-
E_M(\phi^n)
\le
-
\frac1{\Delta t}
\|\phi^{n+1}-\phi^n\|_{H^{-1}}^2.
\]

Therefore,
\[
E_M(\phi^{n+1})
+
\frac1{\Delta t}
\|\phi^{n+1}-\phi^n\|_{H^{-1}}^2
\le
E_M(\phi^n).
\]

\end{proof}

\begin{remark}
In practical computations, one typically observes
$|\phi|\le M$ \cite{OrizagaJCP,Orizaga2024GPU,ShinLeeLee2020}, so that
\[
F_M(\phi)=\frac14\phi^4
\]
throughout the simulation. Thus the truncated formulation serves only as a technical device ensuring global concavity and rigorous unconditional energy stability.
\end{remark}


\section{Computational Stability Analysis}

\subsection{Single Run: Crystal Growth from Seed}

To illustrate the dynamics of the proposed scheme, 
we first present a single simulation of the PFC 
equation initialized with the standard seed 
configuration of Glasner--Orizaga~\cite{OrizagaJCP}. 
The simulation is performed on a $256\times 256$ 
grid with $r=0.5$, $\Delta t = 0.01$, and $a=0.30$, a value notably below the classical threshold 
$a \geq 2$, demonstrating that energy-stable 
and physically correct simulations are achievable 
well below this bound. 
Figure~\ref{fig:snapshots_2D} shows snapshots of 
the atomistic density field at $t=0, 10, 20, 40$, 
demonstrating the progressive growth of a hexagonal 
crystal lattice from the initial seed. 
Figure~\ref{fig:energy_2D} shows the corresponding 
free energy as a function of time, confirming 
monotonic energy dissipation throughout the simulation.

\begin{figure}[h!]
\centering
\begin{tabular}{cc}
\includegraphics[width=0.43\textwidth]{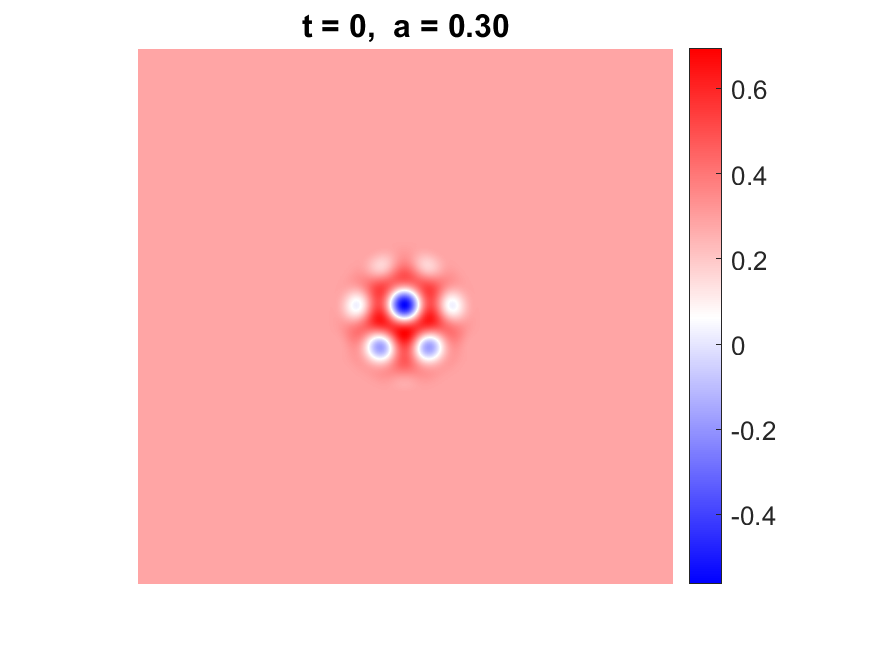} &
\includegraphics[width=0.43\textwidth]{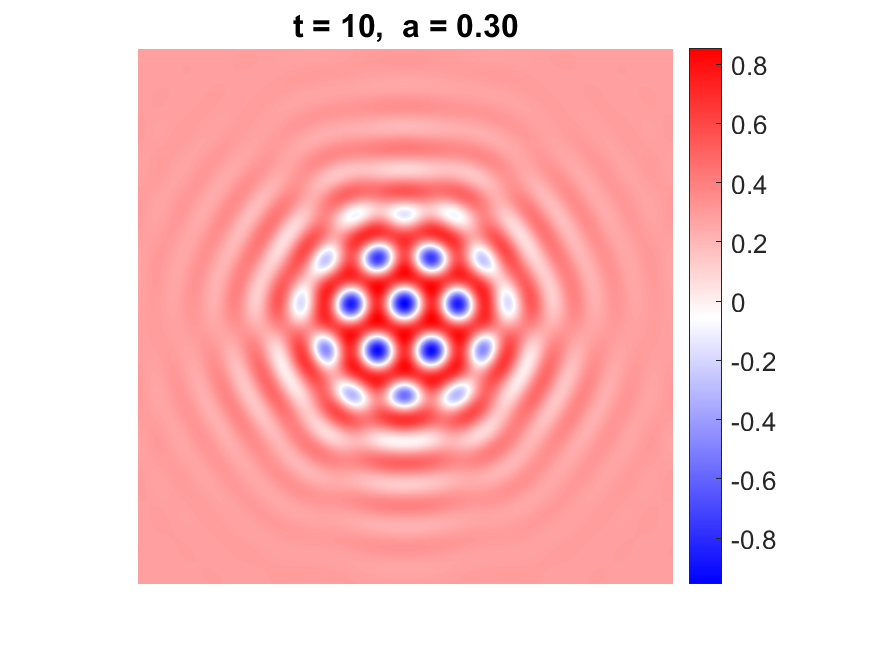} \\
\includegraphics[width=0.43\textwidth]{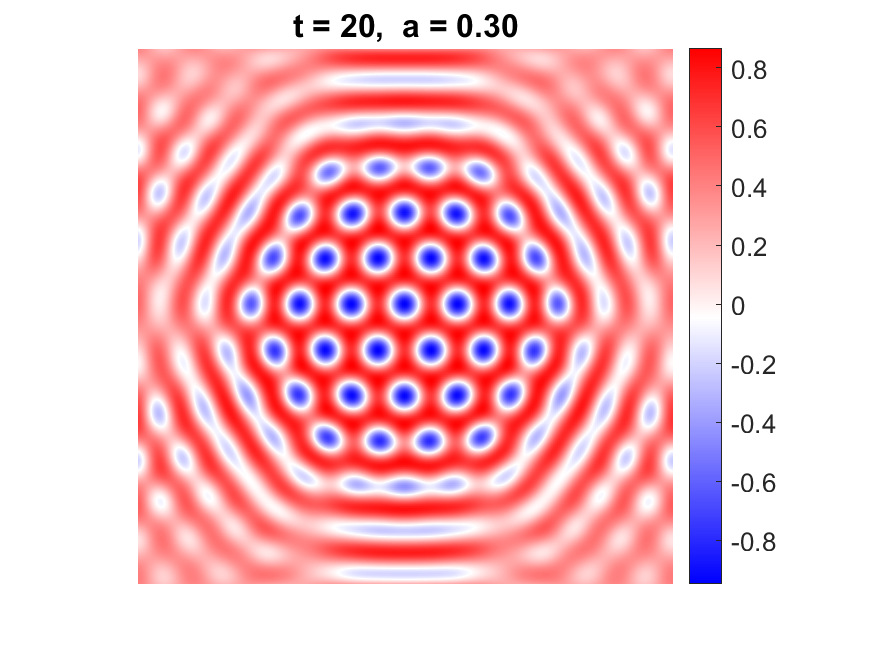} &
\includegraphics[width=0.43\textwidth]{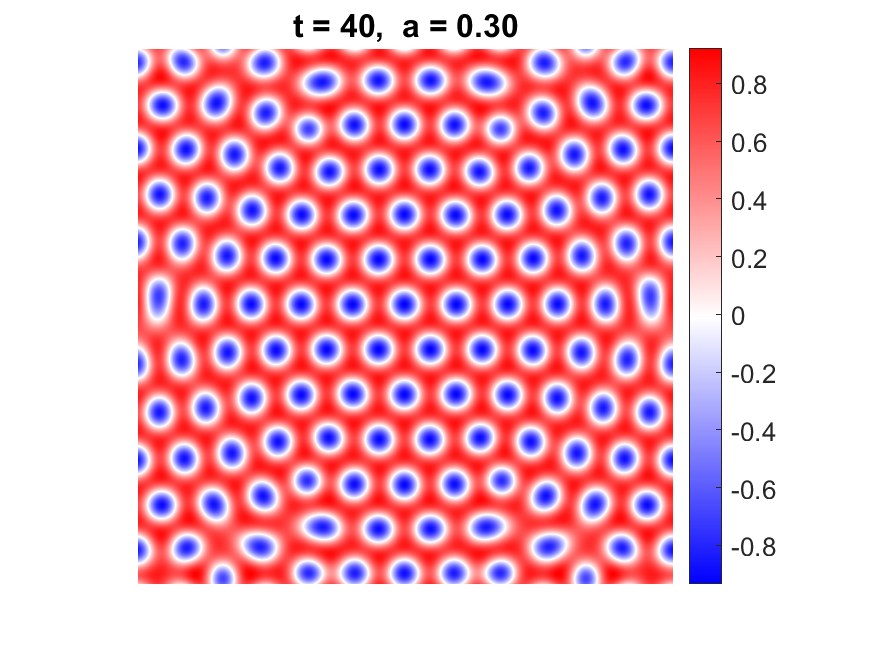}
\end{tabular}
\caption{Crystal growth simulation for the PFC equation 
using the convex splitting spectral method with $a=0.30$, 
$\Delta t = 0.01$, $r=0.5$, on a $256\times 256$ grid. 
Snapshots of the atomistic density field $\phi$ are shown 
at $t=0$ (top left), $t=10$ (top right), $t=20$ (bottom left), 
and $t=40$ (bottom right). The initial seed configuration 
of Glasner--Orizaga~\cite{OrizagaJCP} grows into a fully 
developed hexagonal crystal lattice, demonstrating the 
energy-stable and physically correct dynamics of the 
proposed scheme.}
\label{fig:snapshots_2D}
\end{figure}

\begin{figure}[h!]
\centering
\centerline{\includegraphics[width=1.25\textwidth]{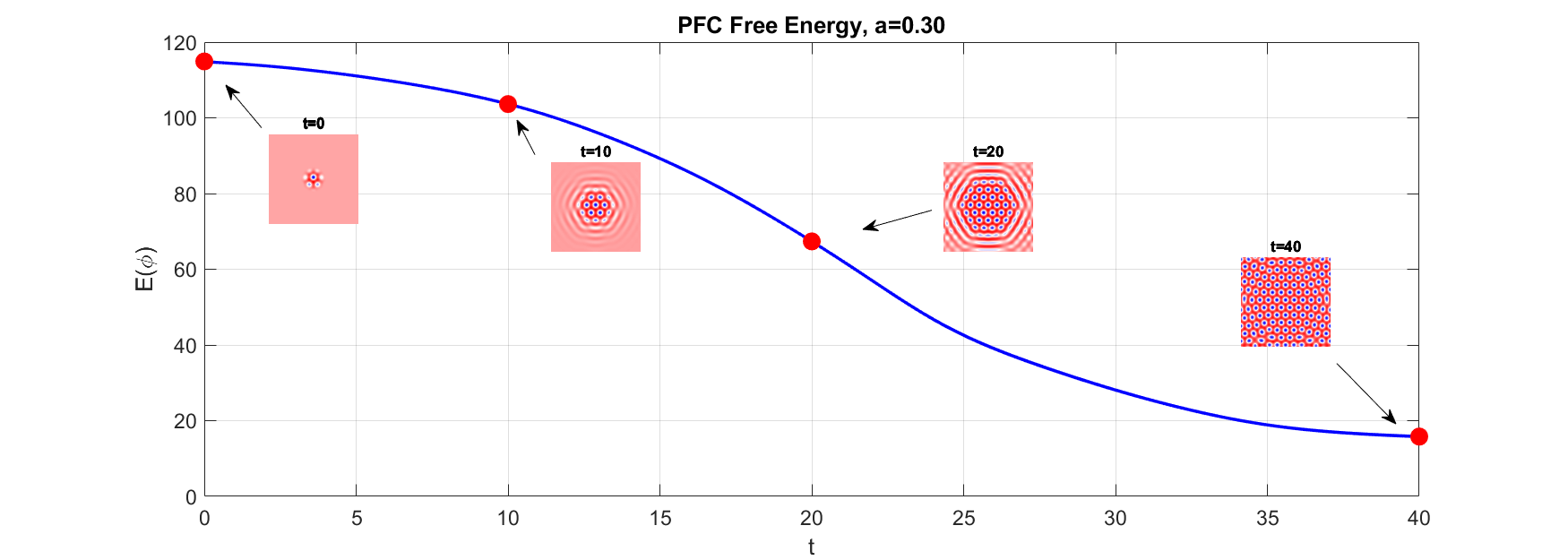}}
\caption{Free energy $E(\phi^n)$ as a function of time 
for the PFC convex splitting scheme with $a=0.30$, 
$\Delta t=0.01$, $r=0.5$, on a $256\times 256$ grid 
up to $t=40$. The energy decreases monotonically 
throughout the simulation, confirming unconditional 
energy stability. Red dots indicate the times 
$t = 0, 10, 20, 40$ at which solution snapshots 
are shown as insets, illustrating the progressive 
development of the hexagonal crystal structure 
from the initial seed configuration.}
\label{fig:energy_2D}
\end{figure}

\subsection{GPU-Accelerated Parallel Simulations}

To further demonstrate the efficiency of the proposed 
framework, we exploit GPU parallelism to execute 
twenty independent PFC simulations simultaneously, 
each with a different value of the splitting parameter 
$a \in [1.00, 2.00]$. This is motivated in part by 
recent advances in GPU-parallelized scientific 
computing, in particular, Chowdhury and 
P\'erez-R\'ios~\cite{chowdhury2026gpu} demonstrated 
the effectiveness of GPU-parallelized MATLAB for 
large-scale parameter sweeps in atom-ion dynamics, 
achieving significant speedups over CPU-based 
implementations. The present work adopts a similar 
philosophy. As shown in Figure~\ref{fig:parallel_20}, 
all twenty simulations are energy-stable, confirming 
that values of $a$ significantly below the classical 
threshold $a \geq 2$ yield physically correct solutions. 
A GPU speedup of $18.42\times$ over the sequential 
CPU baseline is observed.

\begin{figure}[h!]
\centering
\includegraphics[width=0.95\textwidth]{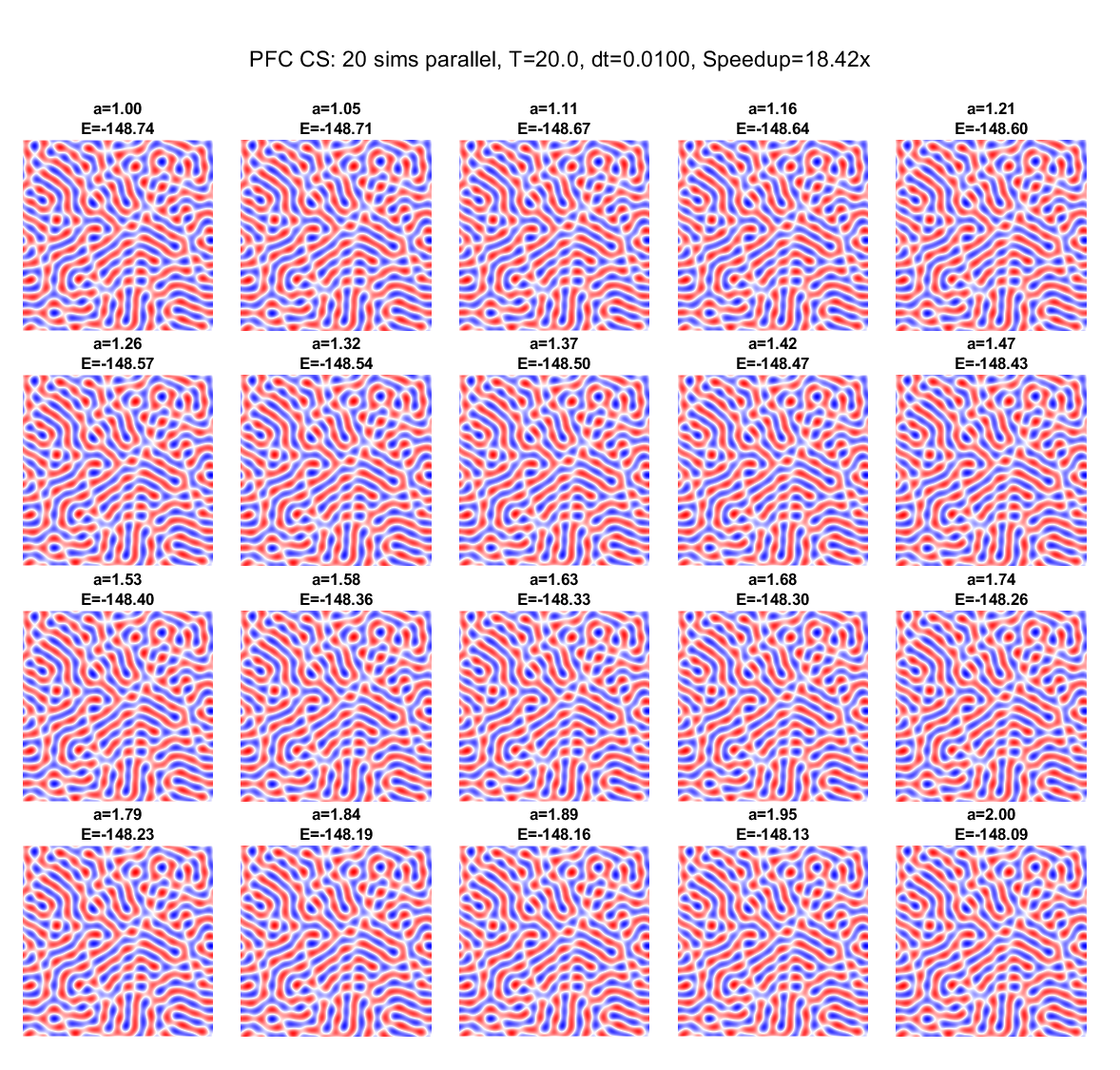}
\caption{Twenty parallel GPU-accelerated PFC simulations 
with splitting parameter $a \in [1.00, 2.00]$, 
$\Delta t = 0.01$, $r = 0.5$, $T = 20$, on a 
$256 \times 256$ grid with random initial condition 
$\phi^0 = 0.01\,\eta$, where $\eta \sim \mathcal{N}(0,1)$. 
All twenty simulations are energy-stable with 
monotonically decreasing free energy, confirming 
that values of $a$ significantly below the classical 
threshold $a \geq 2$ yield physically correct 
and energy-stable solutions. The GPU-accelerated 
parallel implementation achieves a speedup of 
$18.42\times$ over the sequential CPU baseline, 
demonstrating the efficiency of the proposed 
computational framework.}
\label{fig:parallel_20}
\end{figure}

\subsection{Neutral Stability Curve}

While Theorem~\ref{thm:energy} establishes that 
$a \geq L_M$ is sufficient for unconditional energy 
stability, this condition is not sharp in practice. 
To characterize the practical stability boundary, 
we derive a semi-analytical neutral stability curve 
from a dominant-mode energy balance.

From the discrete energy law established in 
Theorem~\ref{thm:energy}, the neutral stability 
transition occurs when $E_M(\phi^{n+1}) = E_M(\phi^n)$, 
that is, when the dissipative terms and the nonlinear 
remainder are in exact balance. To make this condition 
tractable, we invoke a dominant-mode reduction: we project 
the balance onto the dominant Fourier mode 
$|\mathbf{k}|^2 = 2$, evaluate the nonlinearity 
at its bulk-phase value $f'' = F_M''(\phi)|_{\phi=1} = 2$, 
and collect terms in $a$. This yields the quadratic 
condition
\[
a\left(1 + \Delta t\,|\mathbf{k}|^6 
+ \Delta t\,a\,|\mathbf{k}|^2\right) 
= f''\,\Delta t\,|\mathbf{k}|^2.
\]
Solving for the positive root gives the 
analytical neutral curve:
\begin{equation}\label{eq:neutral}
a_{\mathrm{neutral}}(\Delta t) =
\frac{-(1 + \Delta t\,|\mathbf{k}|^6)
+\sqrt{(1+\Delta t\,|\mathbf{k}|^6)^2
+4\,\Delta t^2\,|\mathbf{k}|^4\,f''}}
{2\,\Delta t\,|\mathbf{k}|^2},
\end{equation}
which contains no free parameters under this 
closure. As $\Delta t \to 0$, 
$a_{\mathrm{neutral}} \to 0$, consistent 
with the explicit CFL-type behavior observed 
numerically. For large $\Delta t$, 
$a_{\mathrm{neutral}}$ grows monotonically, 
reflecting increased stabilization requirements. 
The curve lies well below the classical bound 
$a \geq 2$ over the entire tested range 
$\Delta t \in [0.01, 2]$, confirming that the 
standard condition is not sharp in practice.
Any choice satisfying $a \geq a_{\mathrm{neutral}}(\Delta t)$ 
is consistently observed to yield energy-stable 
simulations, providing a practical and conservative 
guideline for parameter selection.

\subsection{GPU-Accelerated Stability Map}

To verify the neutral stability curve~\eqref{eq:neutral} 
computationally, we perform a systematic parameter 
sweep over $(a, \Delta t) \in [0, 0.4] \times [0.01, 2]$ 
using 40,000 independent GPU-accelerated PFC simulations 
on an NVIDIA RTX 5090 GPU. Each simulation is initialized 
with the standard seed configuration in~\cite{OrizagaJCP} on a $64 \times 64$ 
grid with $r = 0.5$. To ensure a fair and rigorous 
comparison across all parameter pairs, each simulation 
is advanced for a fixed number of time steps given by
\[
N_{\mathrm{steps}} = \left\lfloor \frac{T}{\Delta t_{\min}} \right\rfloor 
= \left\lfloor \frac{14}{0.01} \right\rfloor = 1400,
\]
regardless of the value of $\Delta t$. A simulation is 
classified as stable if the discrete free energy satisfies 
$E(\phi^{N_{\mathrm{steps}}}) < E(\phi^0)$, a criterion 
verified to correspond to physically meaningful solutions 
in representative cases.

Figure~\ref{fig:neutral_curve} shows the semi-analytical 
neutral stability curve~\eqref{eq:neutral} overlaid 
against the computational stability boundary extracted 
from the stability map of Figure~\ref{fig:stability_map}. 
The neutral curve lies above the computational boundary 
throughout the tested range, confirming its conservative 
but reliable character as a practical guideline for 
parameter selection. Notably, the neutral curve remains 
substantially below the classical sufficient condition 
$a \geq 2$ of Eyre~\cite{eyre1998unconditionally} 
across the entire tested range $\Delta t \in [0.01, 2]$, 
demonstrating that the classical bound is significantly 
more conservative than necessary in practice.

This approach is broadly applicable to any 
phase-field or gradient flow model requiring 
systematic parameter exploration. The 40,000 
GPU-accelerated simulations presented here 
demonstrate that computational stability maps, 
previously impractical due to cost, are now 
accessible on consumer hardware, opening new 
avenues for parameter sensitivity analysis in 
materials science, crystal growth modeling, 
and beyond.

\begin{figure}[h!]
\centering
\includegraphics[width=0.7\textwidth]{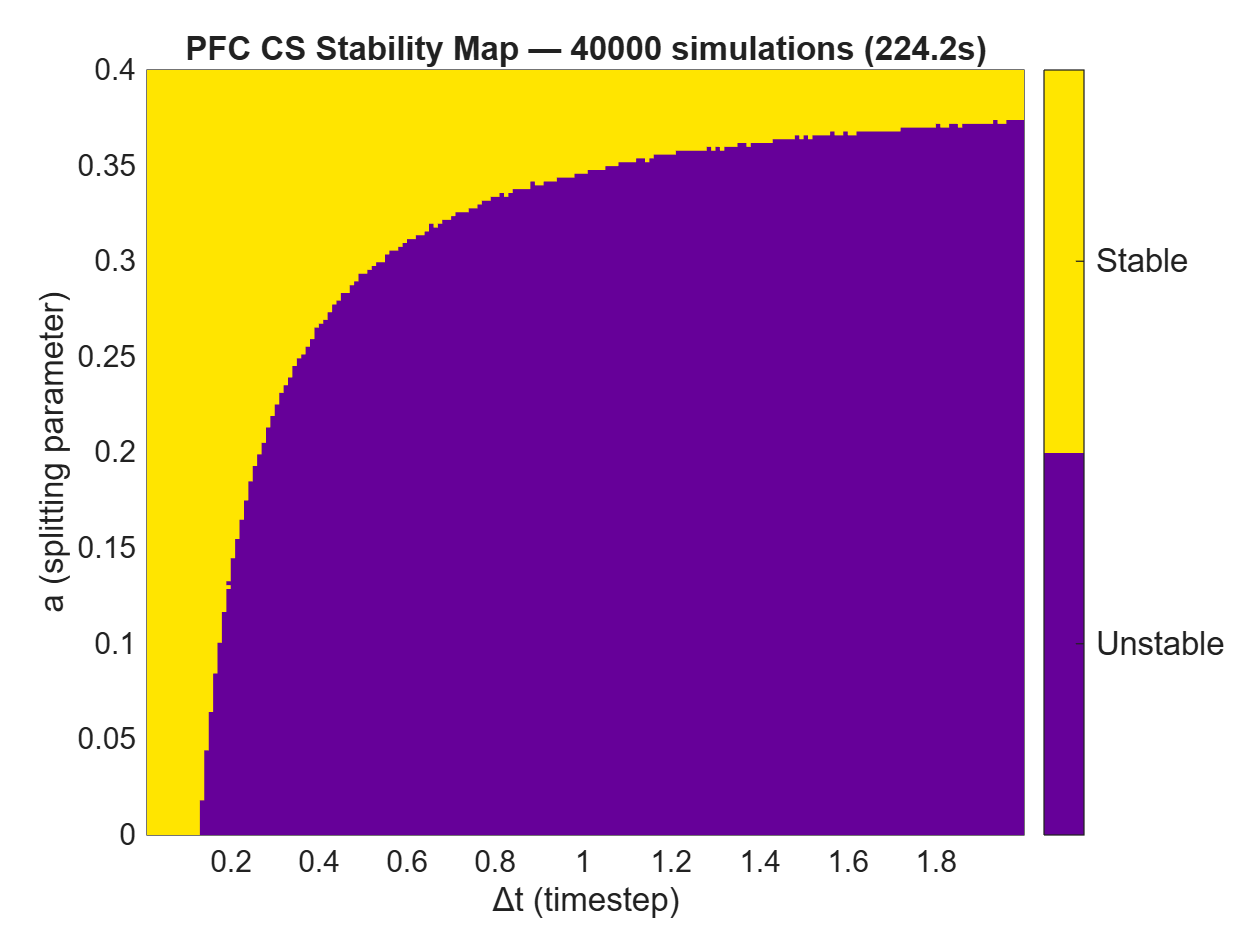}
\caption{Computational stability map for the convex 
splitting scheme applied to the PFC equation, obtained 
from 40,000 GPU-accelerated simulations on an NVIDIA 
RTX 5090 in 224.2 seconds. The parameter space 
$(a, \Delta t) \in [0, 0.4] \times [0.01, 2]$ is 
explored on a $200 \times 200$ grid. Yellow regions 
indicate energy-stable simulations where $E(\phi^n)$ 
decreases monotonically; purple regions indicate 
energy-unstable simulations. The stability boundary is well characterized by 
the semi-analytical neutral stability curve 
$a_{\mathrm{neutral}}(\Delta t)$~\eqref{eq:neutral}, 
lying well below the classical sufficient condition 
$a \geq 2$ of Eyre~\cite{eyre1998unconditionally}. 
This demonstrates that the classical bound is 
sufficient but not sharp, and that energy-stable 
simulations are achievable for values of $a$ 
significantly below this threshold.}
\label{fig:stability_map}
\end{figure}

\begin{figure}[h!]
\centering
\includegraphics[width=0.75\textwidth]{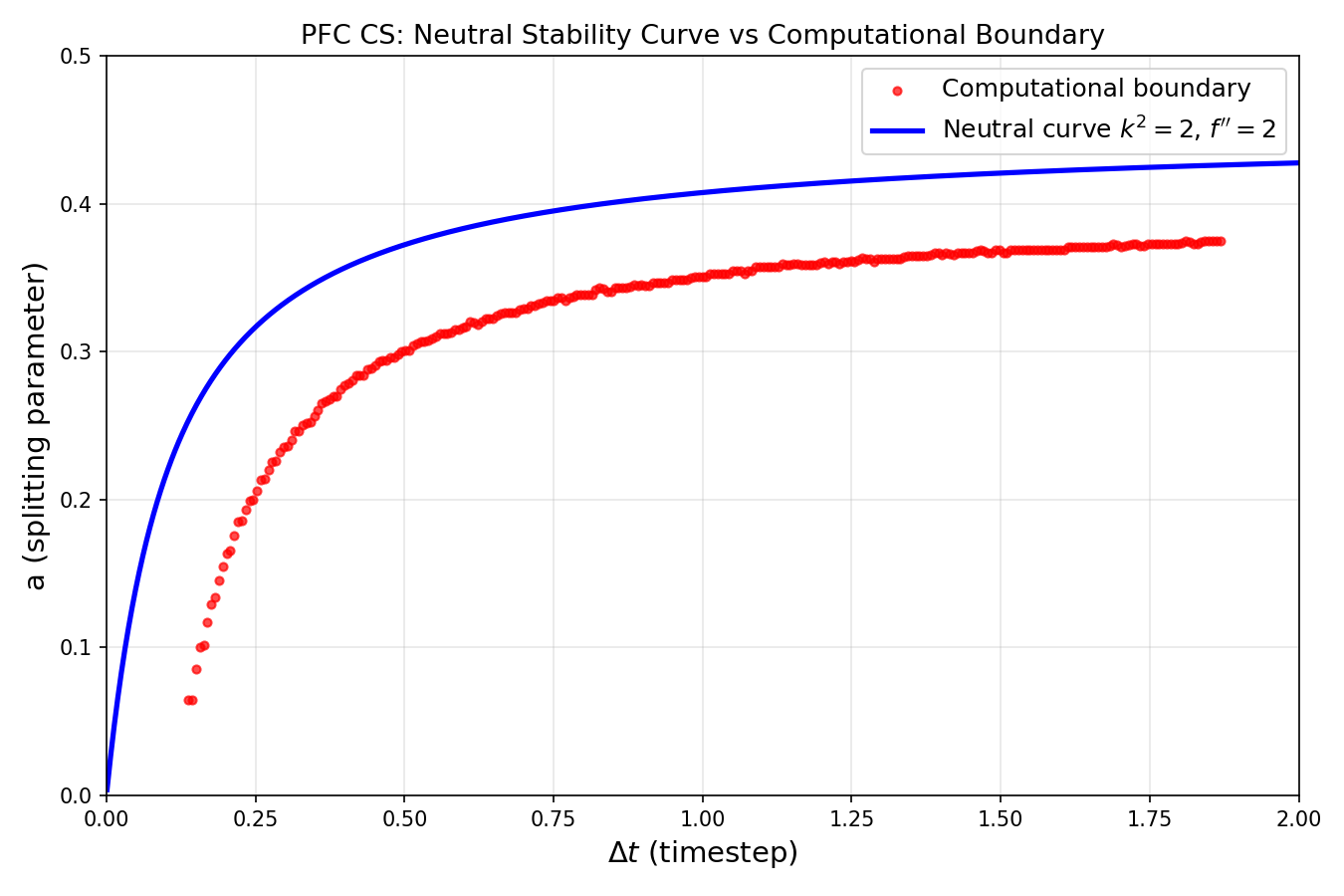}
\caption{Semi-analytical neutral stability curve 
$a_{\mathrm{neutral}}(\Delta t)$~\eqref{eq:neutral} 
(solid blue line) derived from a dominant-mode energy 
balance with $|\mathbf{k}|^2 = 2$ and $f'' = 2$, 
overlaid against the computational stability boundary 
(red dots) extracted from Figure~\ref{fig:stability_map}. 
The neutral curve lies above the computational boundary 
throughout the tested range $\Delta t \in [0.01, 2]$, 
confirming its conservative but reliable character 
as a practical stability guideline.}
\label{fig:neutral_curve}
\end{figure}

\begin{remark}[Robustness of the Stability Map]
The computational stability map presented in 
Figure~\ref{fig:stability_map} was obtained using 
the seed configuration of~\cite{OrizagaJCP}. 
Additional tests with random initial conditions 
$\phi^0 = 0.01\,\eta$, $\eta \sim \mathcal{N}(0,1)$, 
yielded near-identical stability boundaries, 
suggesting that the stability boundary is an 
intrinsic property of the CS scheme rather than 
a consequence of the choice of initial condition.
\end{remark}

\section{Accuracy Analysis: Influence of the Splitting Parameter}

The computational stability map of 
Figure~\ref{fig:stability_map} demonstrates that 
energy-stable simulations are achievable for values 
of $a$ significantly below the classical threshold 
$a \geq 2$. A natural question is whether reducing 
$a$ within the stable region affects the accuracy 
of the numerical solution. We investigate this 
systematically by computing $L^2$ errors relative 
to a reference solution obtained with $a = 2$ and 
$\Delta t = 10^{-5}$ on a $256 \times 256$ grid 
up to $T = 1$, using the standard seed initial 
condition in~\cite{OrizagaJCP}.

Figure~\ref{fig:error_plot} shows the $L^2$ error 
as a function of $\Delta t$ for five values of 
$a \in \{0.1, 0.5, 1.0, 1.5, 2.0\}$. All curves 
exhibit first-order $\mathcal{O}(\Delta t)$ convergence, 
consistent with the first-order temporal discretization. 
Crucially, smaller values of $a$ consistently yield 
lower $L^2$ errors for the same $\Delta t$. This 
demonstrates that selecting $a$ close to 
$a_{\mathrm{neutral}}(\Delta t)$, rather than 
the conservative choice $a = 2$, improves accuracy 
while maintaining energy stability, providing a 
practical guideline for parameter selection.

\begin{figure}[h!]
\centering
\includegraphics[width=0.7\textwidth]{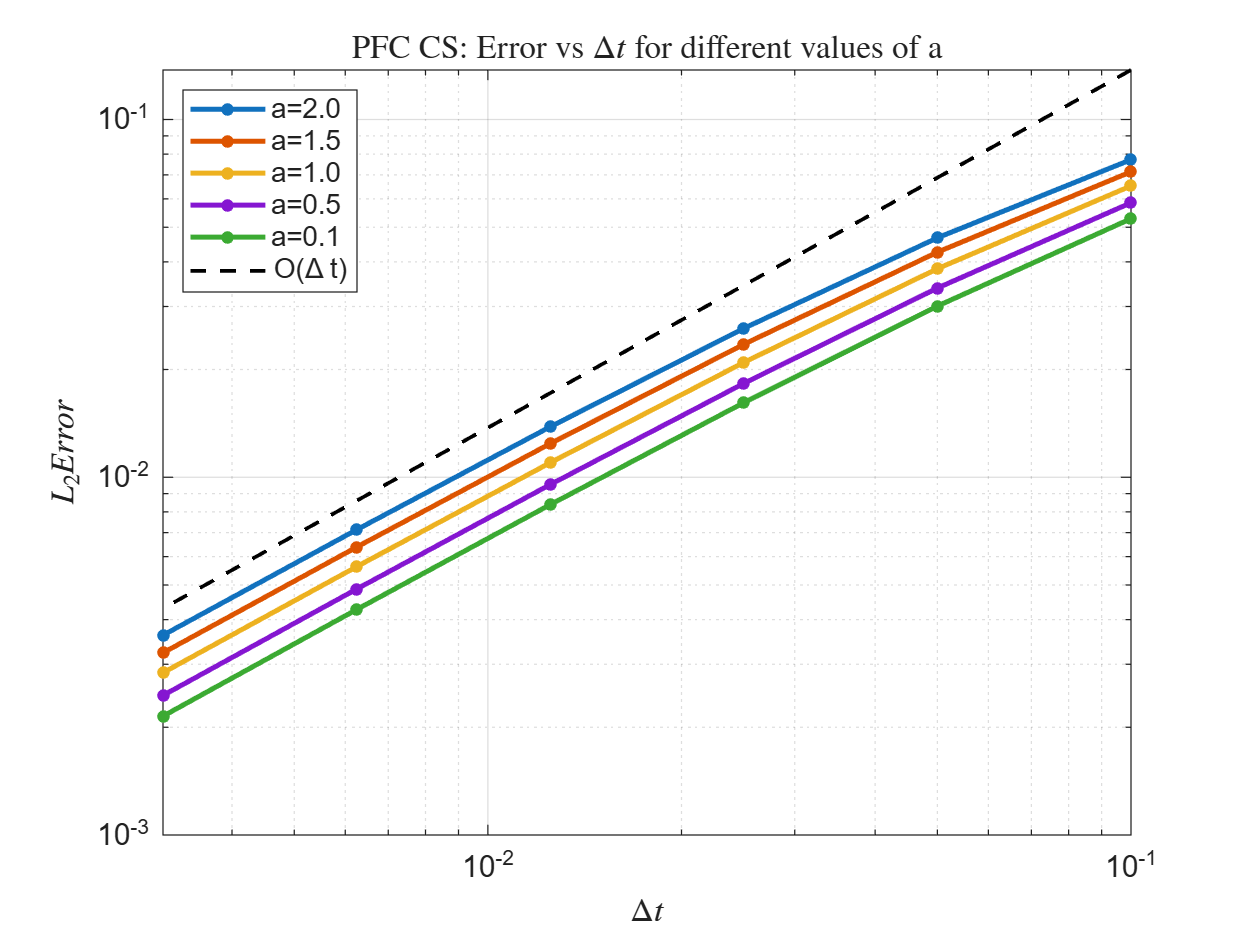}
\caption{$L^2$ error versus $\Delta t$ for the PFC 
convex splitting scheme with five values of the 
splitting parameter $a \in \{0.1, 0.5, 1.0, 1.5, 2.0\}$, 
on a $256\times 256$ grid with $r=0.5$ and $T=1$, 
using the standard seed initial condition of 
Glasner--Orizaga~\cite{OrizagaJCP}. The reference 
solution is computed with $a=2$ and $\Delta t = 10^{-5}$. 
All curves exhibit first-order $\mathcal{O}(\Delta t)$ 
convergence. Smaller values of $a$ consistently yield 
lower $L^2$ errors for the same $\Delta t$, demonstrating 
that selecting $a$ below the classical threshold $a \geq 2$ 
improves accuracy while maintaining energy stability.}
\label{fig:error_plot}
\end{figure}


\section{Numerical Simulations: 2D}
\label{sec:numerics}
\subsection{Lamellar Phase: Stripe Formation}

To further demonstrate the capabilities of the proposed 
scheme beyond the seed configuration, we present 
long-time simulations initialized with a random 
initial condition $\phi^0 = 0.01\,\eta$, 
$\eta \sim \mathcal{N}(0,1)$, on a $256\times 256$ 
grid with $r=0.5$, $\Delta t = 0.01$, and $a=0.5$. 
This configuration tests the ability of the method 
to naturally develop labyrinthine stripe patterns 
from an unstructured initial state, without any 
prescribed geometry. Figure~\ref{fig:lamellar_2D} 
shows snapshots at $t=5$, $t=40$, and $t=1000$, 
illustrating the progressive formation and 
coarsening of the lamellar microstructure. 
Figure~\ref{fig:energy_lamellar} shows the 
corresponding free energy, which decreases 
monotonically and reaches a well-defined plateau 
near $t=500$, confirming that the solution has 
reached a near-equilibrium state.

\begin{figure}[h!]
\centering
\begin{tabular}{ccc}
\includegraphics[width=0.32\textwidth]{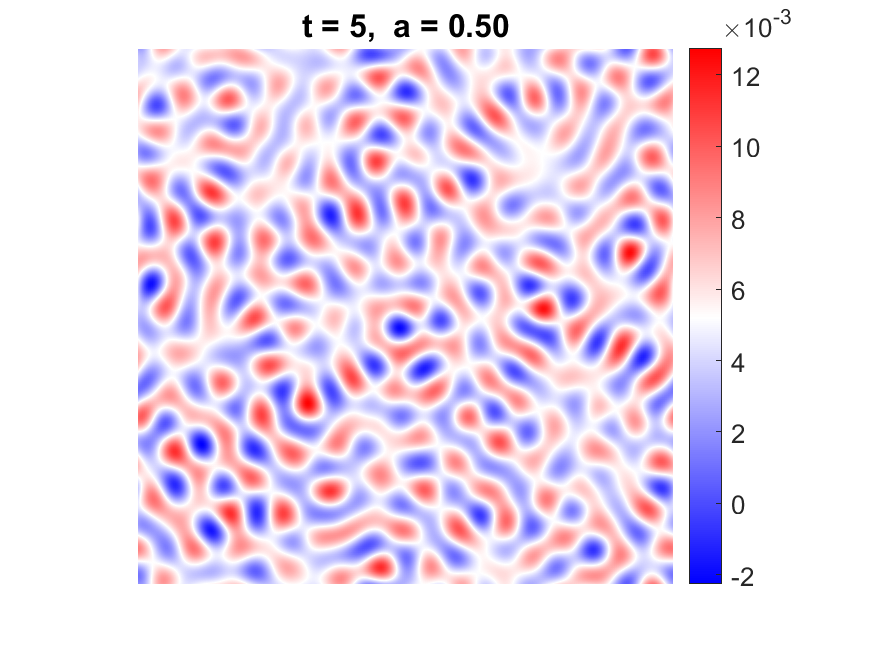} &
\includegraphics[width=0.32\textwidth]{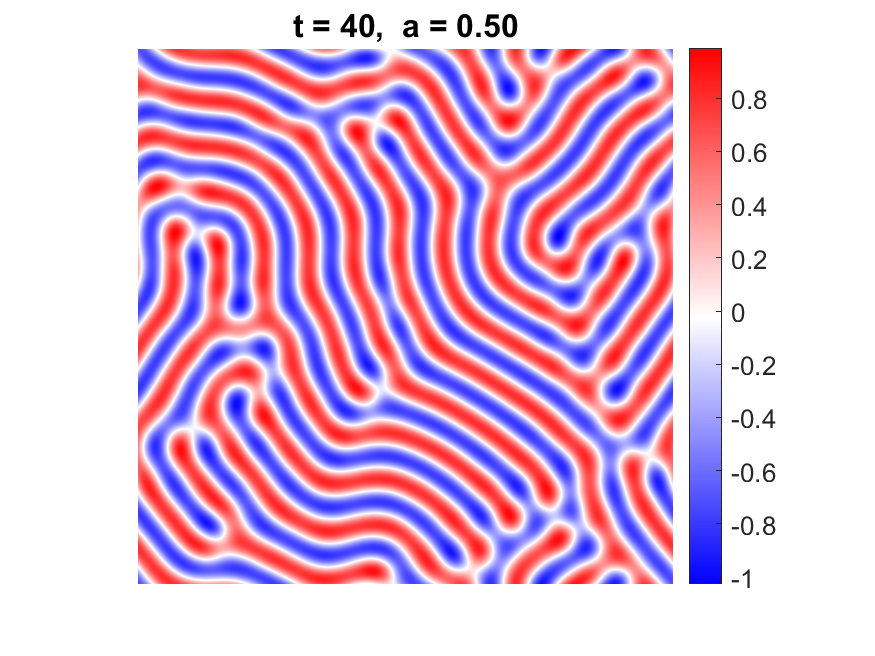} &
\includegraphics[width=0.32\textwidth]{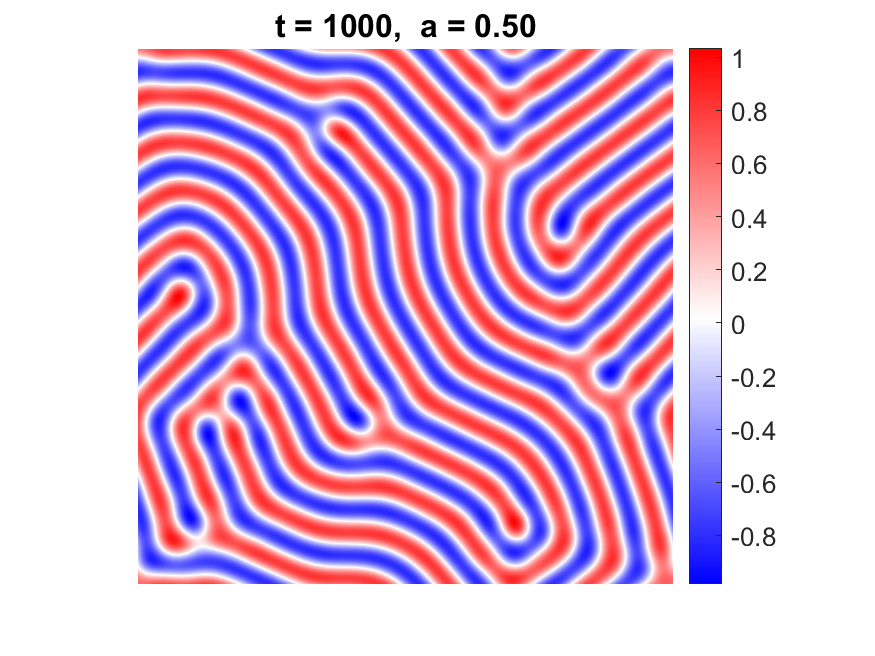}
\end{tabular}
\caption{Labyrinthine stripe pattern formation for the 
PFC equation with random initial condition 
$\phi^0 = 0.01\,\eta$, $\eta \sim \mathcal{N}(0,1)$, 
on a $256\times 256$ grid with $r=0.5$, 
$\Delta t=0.01$, $a=0.5$. Snapshots at $t=5$ (left), 
$t=40$ (center), and $t=1000$ (right) demonstrate 
the progressive development of the labyrinthine 
microstructure from an unstructured initial state 
to a well-developed near-equilibrium configuration.}
\label{fig:lamellar_2D}
\end{figure}

\begin{figure}[h!]
\centering
\includegraphics[width=0.75\textwidth]{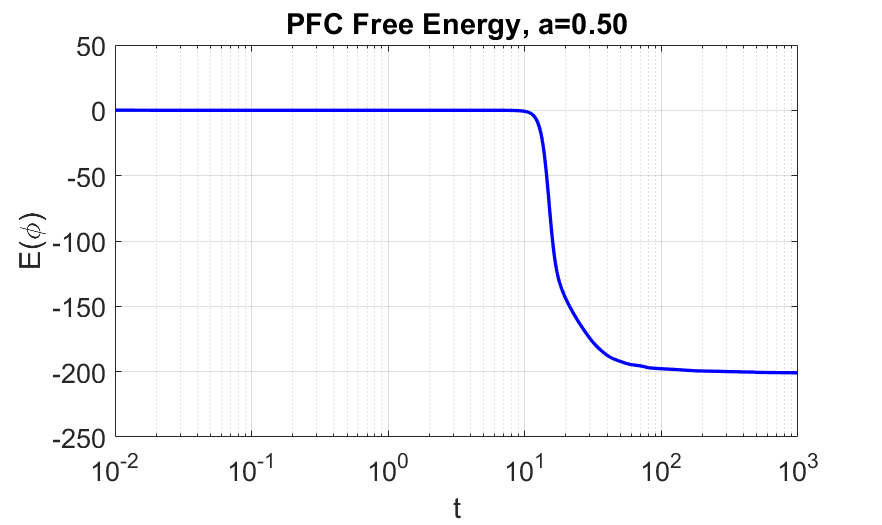}
\caption{Free energy $E(\phi^n)$ as a function of time 
for the PFC simulation of Figure~\ref{fig:lamellar_2D}, 
with random initial condition on a $256\times 256$ grid 
up to $T=1000$. The energy decreases monotonically 
and reaches a well-defined plateau near $t=500$, 
confirming near-equilibrium and unconditional 
energy stability of the proposed scheme.}
\label{fig:energy_lamellar}
\end{figure}


\subsection{Hexagonal Phase}

We next consider a mean-field initial condition 
$\phi^0 = 0.01\,\eta + 0.27$, 
$\eta \sim \mathcal{N}(0,1)$, which favors the 
formation of hexagonal dot patterns. 
Figure~\ref{fig:hex_2D} shows snapshots at 
$t=10$, $t=50$, and $t=1000$, demonstrating 
the progressive development of the hexagonal 
microstructure. Figure~\ref{fig:energy_hex} 
shows the corresponding free energy, which 
decreases monotonically throughout, consistent 
with Theorem~\ref{thm:energy}.

\begin{figure}[h!]
\centering
\begin{tabular}{ccc}
\includegraphics[width=0.32\textwidth]{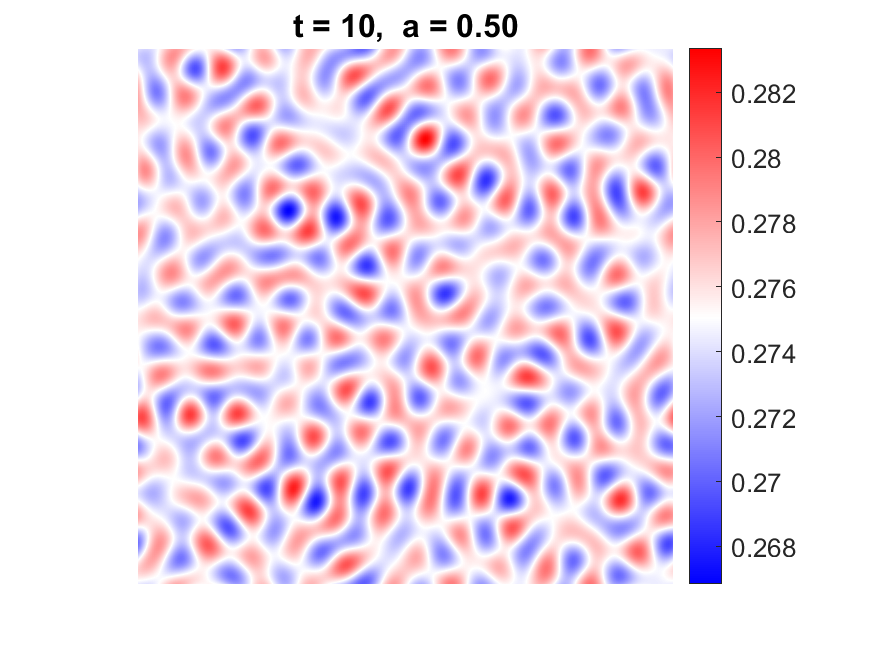} &
\includegraphics[width=0.32\textwidth]{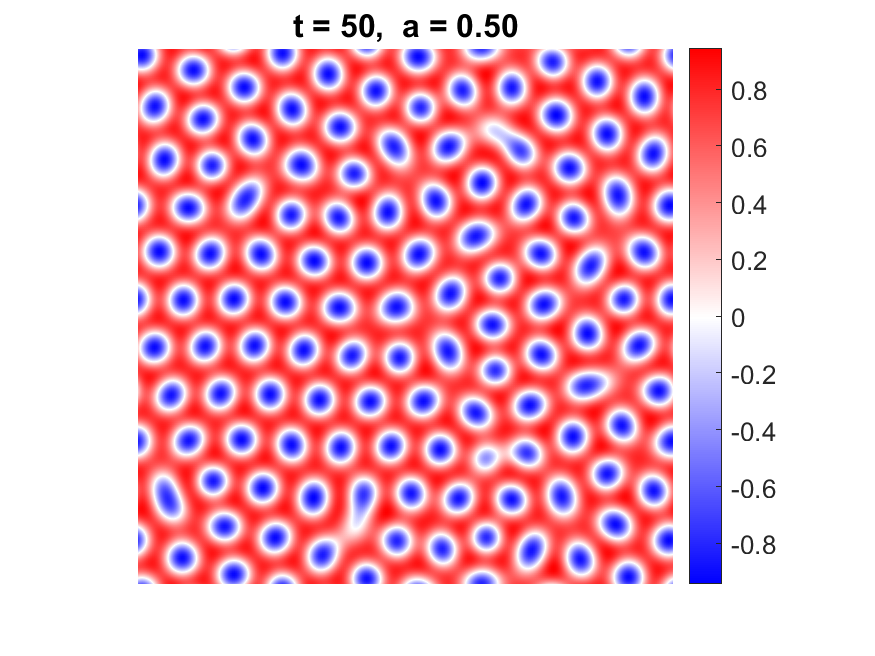} &
\includegraphics[width=0.32\textwidth]{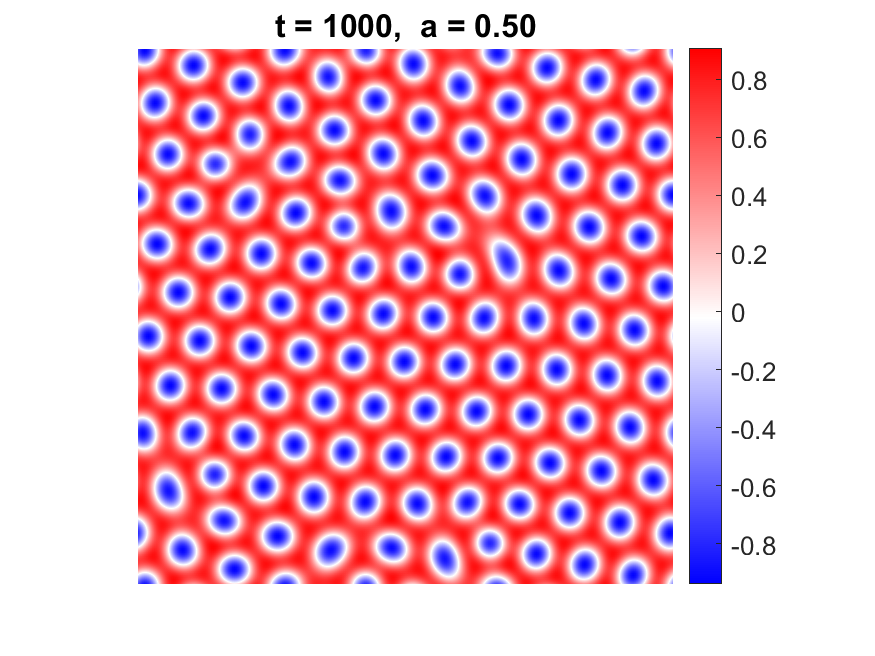}
\end{tabular}
\caption{Hexagonal phase formation for the PFC equation 
with initial condition $\phi^0 = 0.01\,\eta + 0.27$, 
$\eta \sim \mathcal{N}(0,1)$, on a $256\times 256$ 
grid with $r=0.5$, $\Delta t=0.01$, $a=0.5$. 
Snapshots at $t=10$ (left), $t=50$ (center), 
and $t=1000$ (right) demonstrate the progressive 
development of the hexagonal microstructure.}
\label{fig:hex_2D}
\end{figure}

\begin{figure}[h!]
\centering
\includegraphics[width=0.75\textwidth]{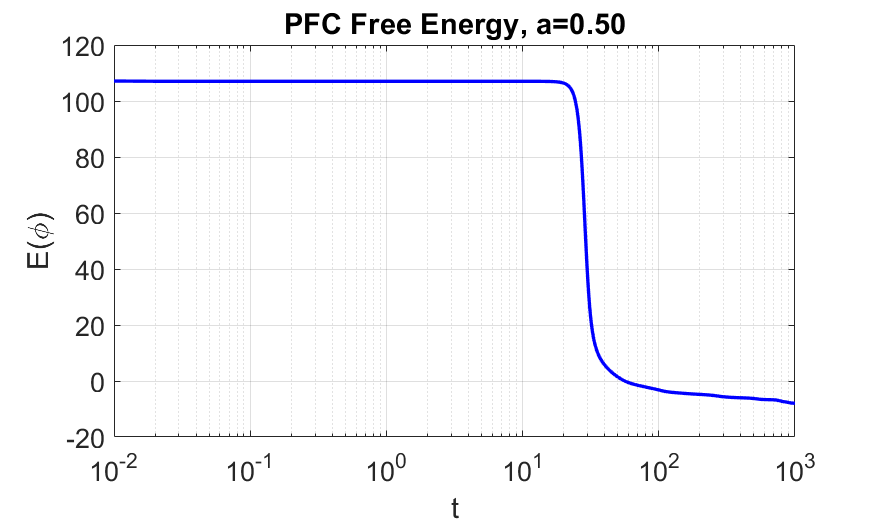}
\caption{Free energy $E(\phi^n)$ as a function of time 
for the hexagonal phase simulation of 
Figure~\ref{fig:hex_2D}, on a $256\times 256$ grid 
up to $T=1000$. The energy decreases monotonically, 
confirming energy stability.}
\label{fig:energy_hex}
\end{figure}

\subsection{Polycrystalline Growth and Grain Boundary Dynamics}

To demonstrate the ability of the proposed scheme 
to capture polycrystalline growth and grain boundary 
dynamics, we initialize three seed crystals with 
different orientations within a $256\times 256$ 
computational domain with $r=0.25$, $q=0.66$, 
$\Delta t = 0.01$, and $a=0.5$. Each seed is 
initialized using the standard crystal nucleus 
configuration ~\cite{OrizagaJCP} 
with a distinct rotation angle, inducing 
crystallographic misorientation between the 
growing grains. As the three crystals grow and 
impinge upon one another, well-defined grain 
boundaries and dislocation defects emerge at 
the interfaces, as shown in 
Figure~\ref{fig:grain_boundary}. The free energy 
decreases monotonically throughout, consistent 
with Theorem~\ref{thm:energy}.

\begin{figure}[h!]
\centering
\centerline{
\begin{tabular}{ccc}
\includegraphics[width=0.34\textwidth]{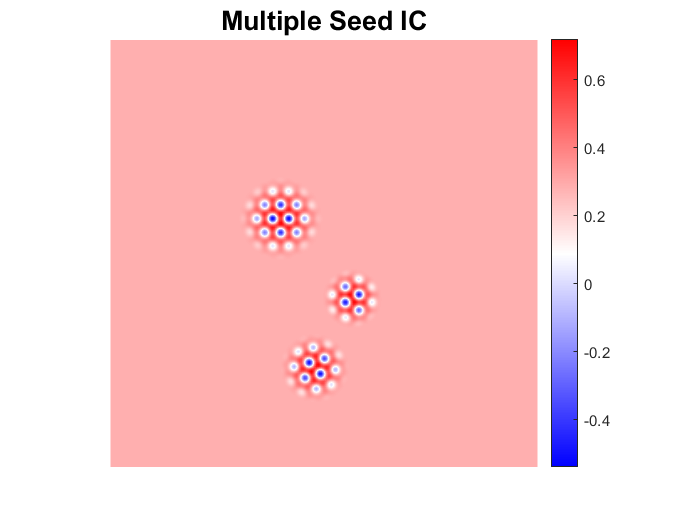} &
\includegraphics[width=0.34\textwidth]{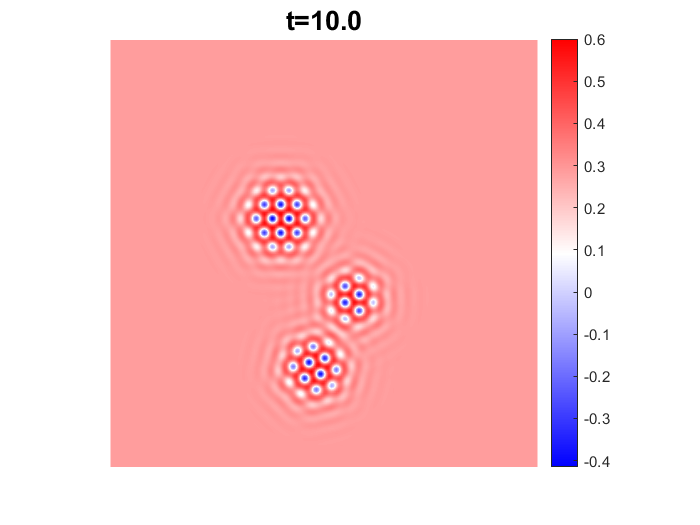} &
\includegraphics[width=0.34\textwidth]{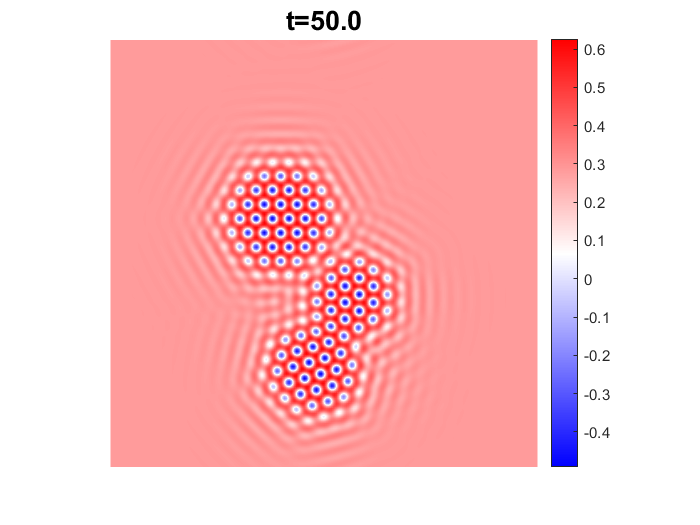} \\[-0.5em]
\includegraphics[width=0.34\textwidth]{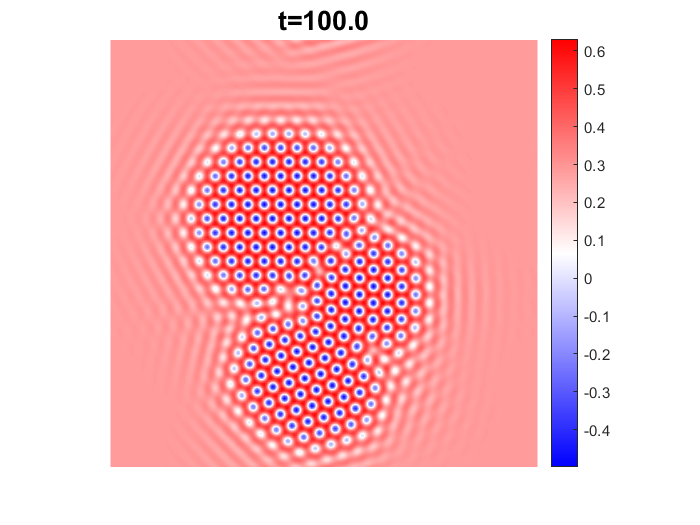} &
\includegraphics[width=0.34\textwidth]{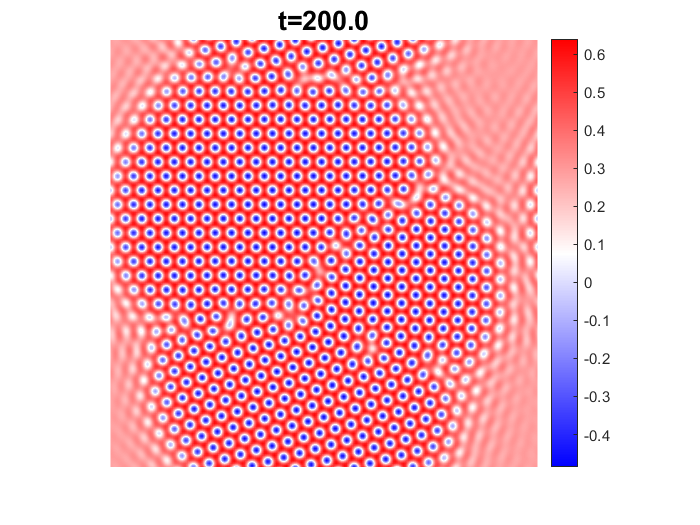} &
\includegraphics[width=0.34\textwidth]{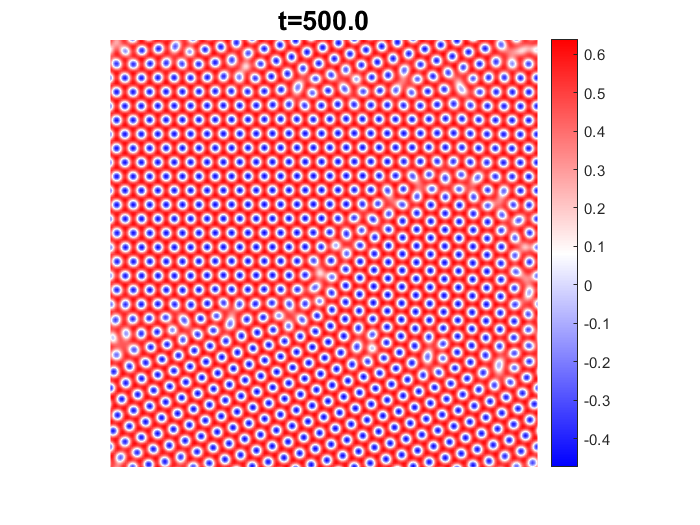}
\end{tabular}
}
\caption{Polycrystalline growth and grain boundary 
dynamics for the PFC equation with three seed crystals 
of different orientations, on a $256\times 256$ grid 
with $r=0.25$, $q=0.66$, $\Delta t=0.01$, $a=0.5$. 
Snapshots at $t=0$, $t=10$, $t=50$ (top row) and 
$t=100$, $t=200$, $t=500$ (bottom row) demonstrate 
the growth of individual crystal grains and their shared 
formation of grain boundaries, as in~\cite{OrizagaJCP}.}
\label{fig:grain_boundary}
\end{figure}


\subsection{Quantitative Grain Boundary Characterization}
\label{sec:grain_quant}

Figure~\ref{fig:grain_boundary} qualitatively demonstrates the
emergence of grain boundaries between three misoriented crystal seeds. Next, we quantify the coarsening dynamics using four types of diagnostic approaches, which were
computed directly from the simulation: the topological defect count
$N_d(t)$, the grain boundary length $L_{GB}(t)$, the grain boundary
excess free energy $E_{GB}(t)$, and the crystallized area fraction
$X(t)$. Results from this anlaysis is presented in Figure \ref{fig:coarsening_diagnostics}.

\begin{figure}[h!]
\centering
\centerline{
\begin{tabular}{cc}
\includegraphics[height=5.5cm]{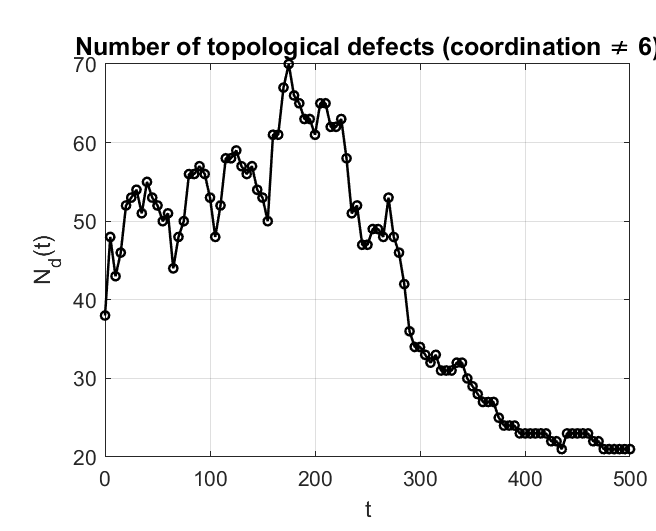} &
\includegraphics[height=5.5cm]{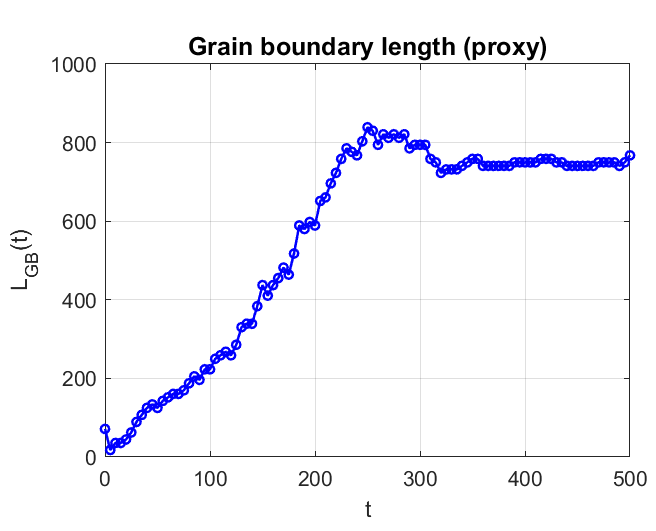} \\[-0.5em]
\includegraphics[height=5.5cm]{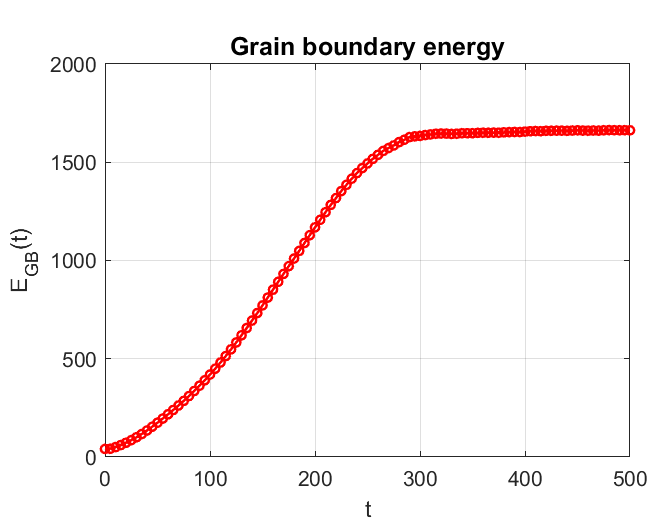} &
\includegraphics[height=5.5cm]{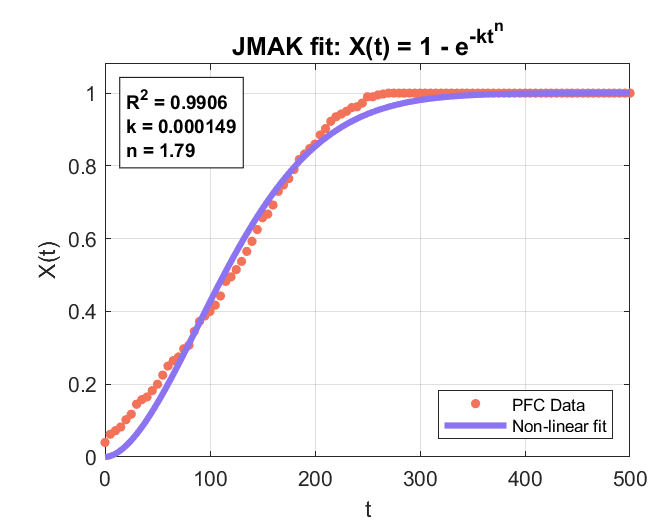}
\end{tabular}
}
\caption{Coarsening diagnostics for the three-grain simulation of
Figure~\ref{fig:grain_boundary}: topological defect count $N_d(t)$,
Eq.~\eqref{eq:Nd} (top left); grain boundary length proxy $L_{GB}(t)$,
Eq.~\eqref{eq:LGB} (top right); grain boundary excess free energy
$E_{GB}(t)$, Eq.~\eqref{eq:EGB} (bottom left); and crystallized area
fraction $X(t)$, fit to the classical Avrami/JMAK form
$X(t)=1-\exp(-kt^n)$ via nonlinear least squares over the full time
series ($k=0.000149$, $n=1.79$, $R^2=0.9906$) (bottom right). All four
quantities plateau at $t^{*}\approx 280$, marking the onset of a
pinned, quasi-steady polycrystalline configuration.}
\label{fig:coarsening_diagnostics}
\end{figure}

\textit{Topological defect count:}
We locate the local extrema of $u(\mathbf{x},t)$, corresponding to the
atomic lattice sites, and construct a Delaunay triangulation over this
point set. In a defect-free triangular lattice every interior site has
coordination number six, so
\begin{equation}
N_d(t) \;=\; \sum_{i \,\in\, \Omega_{\mathrm{int}}} \mathbb{1}\!\left[\, z_i(t) \neq 6 \,\right],
\label{eq:Nd}
\end{equation}
where $z_i(t)$ denotes the Delaunay coordination number of site $i$, $\mathbb{1}[\cdot]$ is the indicator function that equals $1$ if the enclosed condition is satisfied and $0$ otherwise, and $\Omega_{\mathrm{int}}$ consists of all lattice sites located at least one lattice spacing away from the domain boundary.

We emphasize that $N_d(t)$ aggregates all anomalous-coordination sites
across the domain; meaning, \eqref{eq:Nd} captures deviations from ideal coordination number equal to six. Comparing with Figure \ref{fig:grain_boundary}, such deviations appear at the crystal/liquid interfaces (e.g., $t$=100.0), and grain boundary regions (e.g., $t$=500.0).

The results from $N_d(t)$ \textit{v.s.} $t$ analysis is presented in the top-left panel of Figure \ref{fig:coarsening_diagnostics}. Broadly, this plot shows three distinct regions: An initial rise between $0 \leq t \leq 200$, monotonic decrease from $t \approx 200$ to $400$, and a near steady-state till the end of simulation, i.e., $t = 500$. The initial rise to $t \approx 200$ corresponds to increasing fraction of anomalously coordinated atoms at crystal/liquid interface during growth (while largely surrounded by the liquid phase), and those at the grain boundaries formed by the impinging crystallites (e.g., see panel corresponding to $t=100$ in Figure \ref{fig:grain_boundary}. Next, as the liquid is continually consumed by growing crystals, the anomalously coordinated atom population is increasingly dominated by those at the grain boundaries.  For example, the $t=200$ snapshot indicated that the microstructure is dominated by grain boundaries in the second stage. Interestingly, it appears that such grain boundary atoms were in smaller quantity than those at the crystal/liquid interface. Furthermore, our analysis suggested that microstructure did not stop evolving even after complete crystallization, e.g. to $t=400$. (Crystallization kinetics captured from our simulations will be discussed latter). However, beyond this point, the grain boundaries appear to be "locked in" an equilibrium state, which is reflected as a steady state in the plot.

\textit{Interface length:}
We further characterized these interfaces by computing local orientation, $\theta(\mathbf{x},t)$, as a function of time. This quantity was obtained from the
dominant spectral peak of a sliding-window Fourier transform of
$u(\mathbf{x},t)$, taken modulo $60^{\circ}$. Writing
$\zeta(\mathbf{x},t) = \exp[i\,6\,\theta(\mathbf{x},t)]$ to avoid the
wrap-around artifact of this periodicity,
\begin{equation}
L_{GB}(t) \;=\; \int_{\Omega} H\!\left(\, \tfrac{1}{6}\left|\nabla \zeta(\mathbf{x},t)\right| - \eta \,\right) \, d\mathbf{x},
\label{eq:LGB}
\end{equation}
with $\eta = 5^{\circ}$ per grid step, is a coarse but reproducible
proxy for the grain-boundary arclength. 

The top-right plot in Figure \ref{fig:coarsening_diagnostics} shows the result of our analysis. We find that the interfacial length steadily increases till $t\approx250$, and, subsequently, acquires a steady-state value. A small peak was also noted within $200\leq t \leq 300$. It was either caused by the transition to a grain boundary dominated microstructure or is an imprint of the numerical procedure employed in this analysis. Notwithstanding, the initial rise in $L_{GB}(t)$  contains contributions from the increasing circumference of the growing crystals and the proto grain boundaries, e.g. see snapshots corresponding to $t=50$ and $100$ in Figure \ref{fig:grain_boundary}. On the other hand, the steady-state is primarily dominated by the grain boundaries after $t\approx300$.

\textit{Interfacial excess free energy:}
With bulk reference $f_{\mathrm{bulk}}(t) = \mathrm{median}_{\mathbf{x}\in\Omega}\, f(u(\mathbf{x},t))$,
\begin{equation}
E_{GB}(t) \;=\; \int_{\Omega} \Big[\, f\big(u(\mathbf{x},t)\big) - f_{\mathrm{bulk}}(t) \,\Big]^{+} \, d\mathbf{x}
\label{eq:EGB}
\end{equation}
computes the excess free energy  of interfaces, and this approach was inspired by Cahn-Hilliard-based formulations\cite{cahn1958free, granasy2000cahn1,galenko2024hodograph,hasan2025metastable}. Plot in the bottom-right panel of Figure \ref{fig:coarsening_diagnostics} depicts the evolution of  $f_{\mathrm{bulk}}(t)$, and it trends comparable to $L_{GB}(t)$ (top-right panel): a monotonic rise followed by steady state. Therefore, the initial rise correspond to excess energies of crystal/liquid and grain boundary interfaces, while the steady-state reflects contributions from the interfaces formed by the misoriented crystals,i.e., grain boundaries. 

\textit{Crystallization kinetics:}
Lastly, we interrogated the crystallization kinetics by measuring the crystallized fraction ($X(t) \in [0,1]$ ) as function of time - bottom-right panel in Figure \ref{fig:coarsening_diagnostics}. We note that trend in $X(t)$ is comparable to $L_{GB}(t)$ (top-right panel) and $E_{GB}(t)$ (bottom-left panel), and that full crystallization ($X(t)=1)$) was achieved around $t\approx250$. The latter corresponded to microstructure dominated by grain boundary interfaces. Next, we quantified crystallization kinetics by fitting the Johnson–Mehl–Avrami–Kolmogorov (JMAK) equation  to the simulation data \cite{kostorz1983physical,kolmogorov1937statistical}:
\begin{equation}
X(t) = 1 - \exp(-k t^n)
\label{eq:JMAK}
\end{equation}
where, $k$ is a measure of the speed of liquid$\xrightarrow{}$solid (crystal) transformation, and the exponent $n$ signifies the dimensionality of transformation. $n\approx3$ corresponds to three dimensional volumetric transformation, while $n\approx2$ points to a transformation mechanism that is dominated by interfaces \cite{kostorz1983physical,kolmogorov1937statistical,hasan2024transformation}. The fitting results presented in the plot, and, importantly, we find that $n=1.79$ - close to a value of two. This further confirms that the microstructural evolution observed in Figure \ref{fig:grain_boundary} is dominated by interfaces; first by the crystal/liquid, and, after complete crystallization, by crystal/crystal grain boundary interfaces.

Having established the scheme's ability to capture and quantify
mesoscale polycrystalline dynamics, we next examine its performance
at the atomistic length scales for which the PFC framework was
originally intended.


\subsection{Atomistic Resolution: High-Resolution Crystal Growth}

The PFC equation was originally derived to model
microstructure evolution at atomic length 
scales~\cite{Elder,ElderGrant2004}. As demonstrated 
in the comprehensive review of Gránásy et 
al.~\cite{GRANASY2019100569}, phase-field crystal 
models require sufficient spatial resolution to 
faithfully capture crystal nucleation, grain 
boundary dynamics, and microstructure evolution 
at the molecular scale, phenomena that are 
inaccessible at coarser resolutions. To demonstrate 
the ability of the proposed scheme to resolve such 
dynamics, we consider high-resolution simulations 
on grids up to $2048\times 2048$, enabling thousands 
of individual crystalline sites to be resolved 
simultaneously and directly capturing the atomistic 
lattice structure characteristic of the model.

To further assess the behavior of the scheme in a strongly
under-resolved stability regime, we perform a simulation on a
$2048\times 2048$ grid using $a=0.25$, well below the classical
convex-splitting threshold $a\geq 2$. Despite this, the computation
remains stable throughout and accurately captures the evolution of
thousands of crystalline sites across the domain.
Figure~\ref{fig:2048} shows the resulting dynamics, illustrating the
formation and coarsening of a hexagonal crystal lattice from a
single seed.

These large-scale simulations highlight that the relaxed stability
conditions identified in the computational analysis remain effective
in practice for long-time, high-resolution atomistic dynamics.
This sets the stage for extending the framework to fully three-dimensional
simulations in the following section.

\begin{figure}[h!]
\centering
\begin{tabular}{cc}
\includegraphics[width=0.43\textwidth]{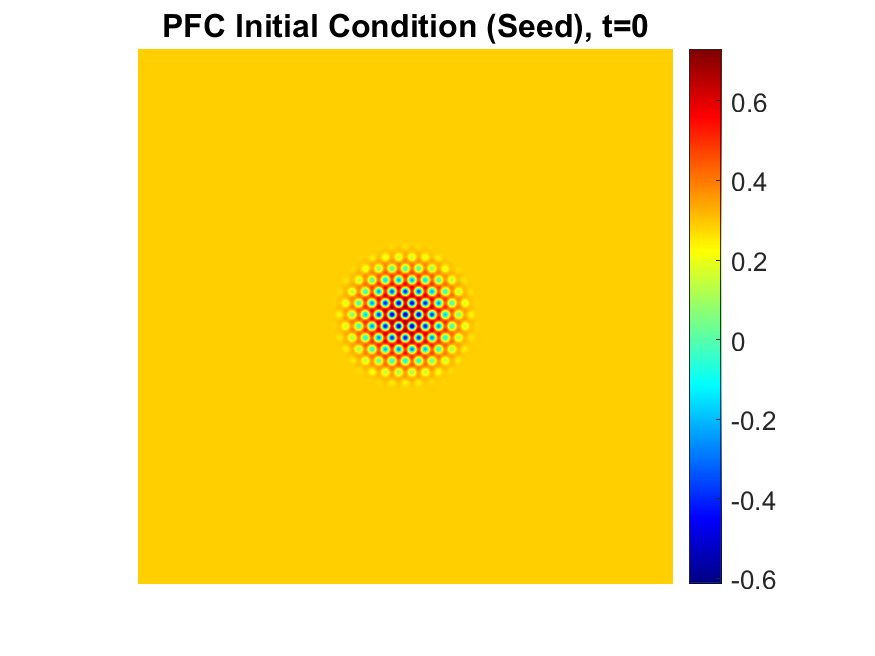} &
\includegraphics[width=0.43\textwidth]{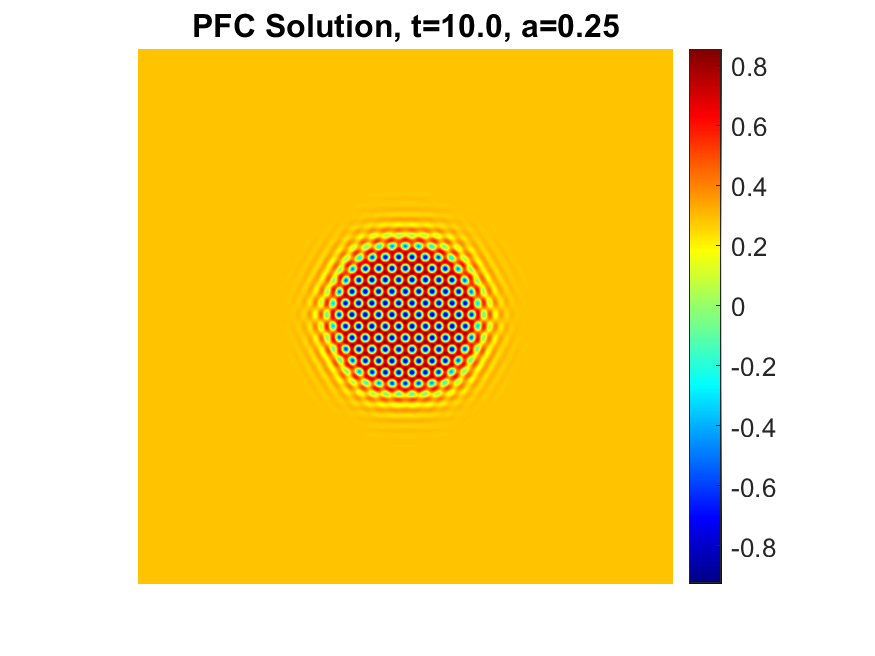} \\
\includegraphics[width=0.43\textwidth]{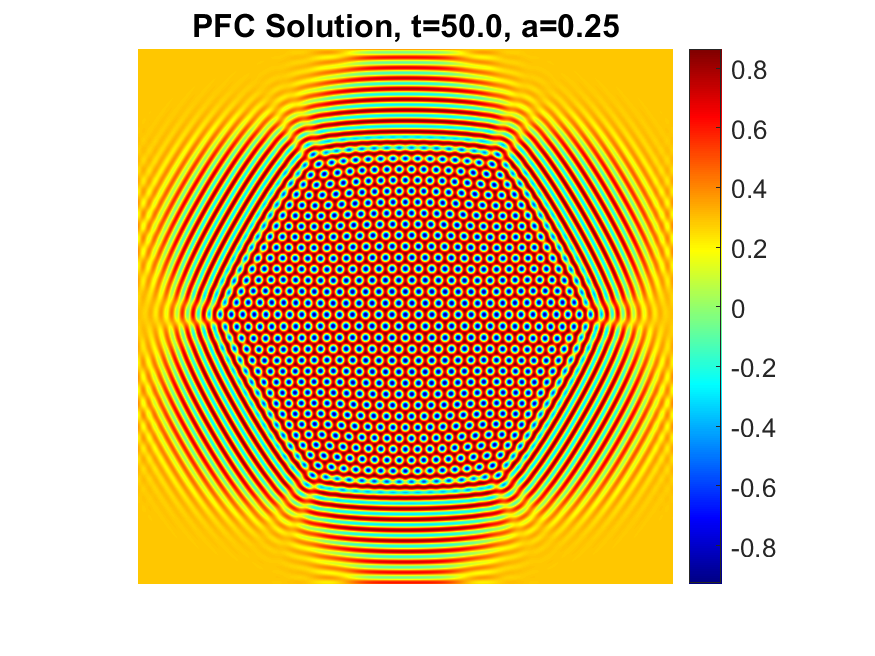} &
\includegraphics[width=0.43\textwidth]{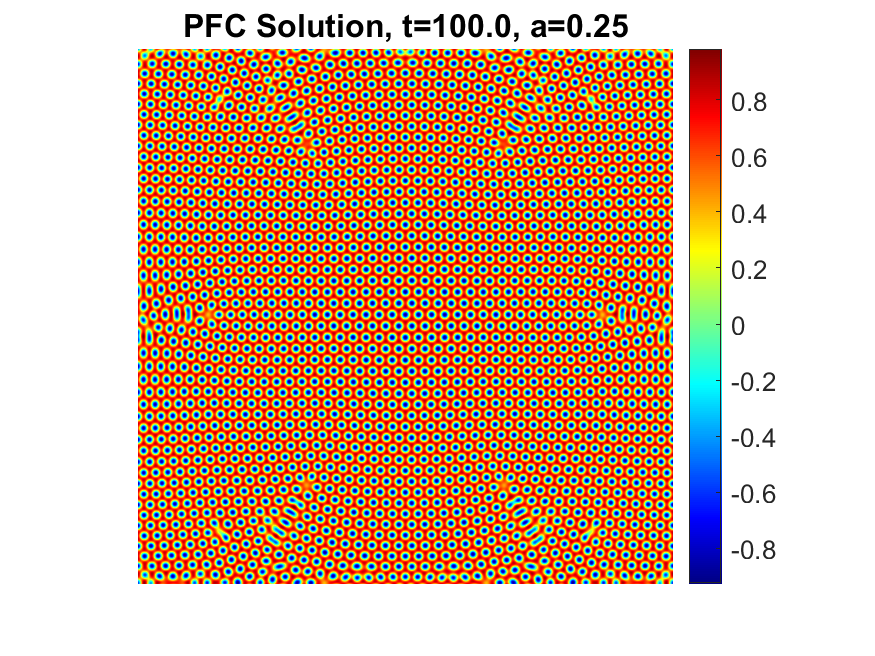}
\end{tabular}
\caption{Crystal growth simulation for the PFC equation 
on a $2048\times 2048$ grid with $a=0.25$, 
$\Delta t = 0.01$, $r=0.5$, well below the classical 
threshold $a\geq 2$. Snapshots at $t=0$ (top left), 
$t=10$ (top right), $t=50$ (bottom left), and $t=100$ 
(bottom right) reveal the progressive growth of the 
hexagonal crystal lattice from a single seed, with 
atomistic-scale microstructure and interfacial wave 
patterns fully resolved at this resolution.}
\label{fig:2048}
\end{figure}

\section{Three-Dimensional Simulations and GPU Scalability}

\subsection{GPU Scalability}

GPU acceleration for spectral phase-field models 
was previously benchmarked in~\cite{Orizaga2024GPU}, 
where an NVIDIA RTX 2060 GPU achieved a conservative 
$8\times$ speedup over an Intel Core i9 CPU baseline 
for the Cahn--Hilliard equation on large-scale problems 
involving $256^3$ grids. GPU-accelerated spectral methods 
for PDEs have also been explored in the context of 
tensor product operations~\cite{ShenGPU} and 
thin-film equations~\cite{KondicGPU}, demonstrating 
the broad applicability of GPU acceleration for 
spectral phase-field computations. In the present 
work, GPU parallelism is further exploited to perform 
large-scale parameter sweeps involving 40,000 
independent PFC simulations, motivated by recent 
advances in GPU-parallelized scientific 
computing~\cite{chowdhury2026gpu}.

To illustrate the three-dimensional capabilities 
of the proposed scheme, we present a benchmark 
simulation using the standard seed initial condition 
of~\cite{OrizagaJCP} extended to three spatial 
dimensions. The benchmark is designed to reflect 
realistic computational conditions: a main script 
calls the solver, timing begins at initialization, 
and concludes only once the full three-dimensional 
solution is computed and available for visualization 
and exploration. All simulations use $r=0.5$, 
$\Delta t=0.01$, $a=0.5$, and $T=100$, 
corresponding to 10,000 time steps at resolutions 
of $128^3$ and $256^3$, representing approximately 
2 million and 16 million degrees of freedom per 
time step, respectively. Figure~\ref{fig:3D_benchmark} 
shows the isosurface visualization and internal 
cross-sections at the center of the domain, 
revealing the three-dimensional crystal 
microstructure. Mass is conserved to machine 
precision throughout.

\begin{figure}[h!]
\centering
\begin{tabular}{cc}
\includegraphics[width=0.5\textwidth]{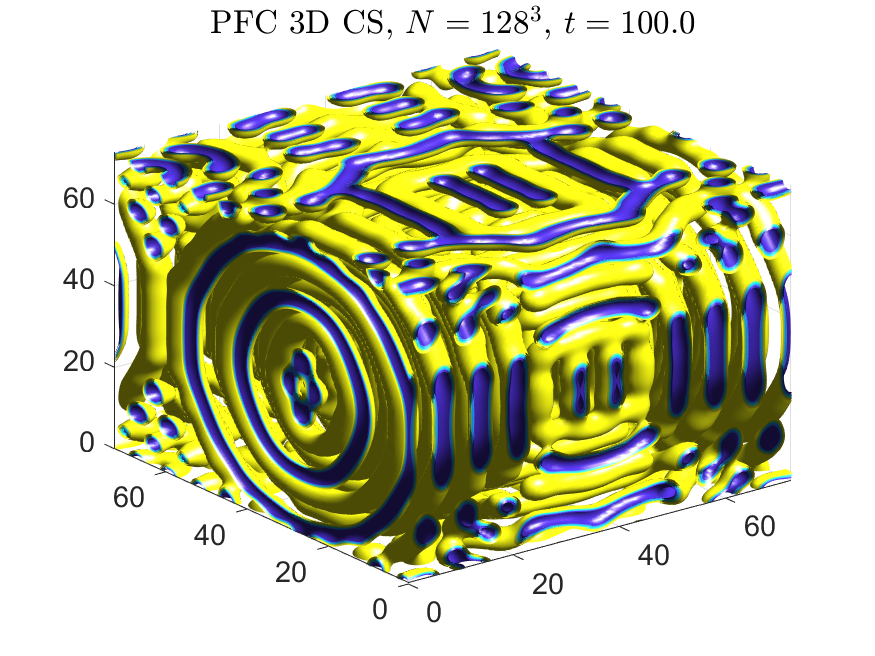} &
\includegraphics[width=0.5\textwidth]{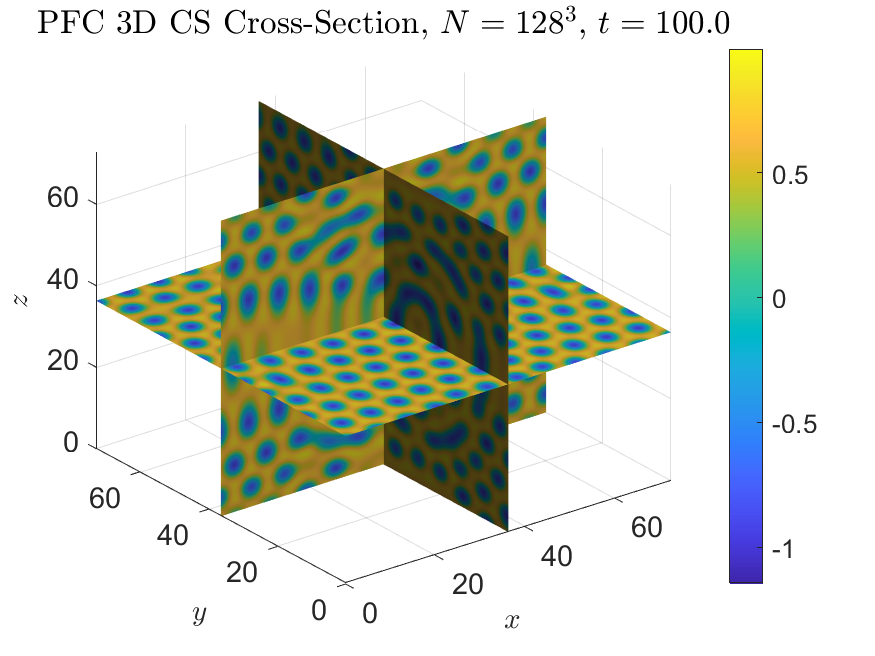}
\end{tabular}
\caption{Three-dimensional PFC simulation using the 
convex splitting spectral method with the seed 
initial condition of~\cite{OrizagaJCP} extended 
to 3D, on a $128^3$ grid with $r=0.5$, 
$\Delta t=0.01$, $a=0.5$, $T=100$. 
Left: isosurface visualization. Right: internal 
cross-sections at the center of the domain.}
\label{fig:3D_benchmark}
\end{figure}

Building on this foundation, we benchmark three 
GPU configurations for the PFC equation using 
the same seed initial condition with $\Delta t = 0.01$ 
and $T = 100$. Table~\ref{tab:gpu} summarizes 
the wall-clock times and relative speedups for 
grid resolutions of $128^3$ and $256^3$.

\begin{table}[h!]
\centering
\begin{tabular}{lcccc}
\toprule
GPU & $128^3$ (min) & $256^3$ (min) & Speedup ($128^3$) & Speedup ($256^3$) \\
\midrule
NVIDIA RTX 2060           & 7.29 & 56.24 & 1$\times$     & 1$\times$ \\
NVIDIA RTX 5090 (laptop)  & 1.42 & 11.82 & 5.13$\times$  & 4.75$\times$ \\
NVIDIA RTX 5090 (desktop) & 0.85 & 6.41  & 8.58$\times$ &  8.77$\times$ \\
NVIDIA RTX PRO 6000 & 0.68 & 4.77  & 10.72$\times$ &  11.79$\times$ \\
\bottomrule
\end{tabular}
\caption{Wall-clock times (in minutes) and GPU speedups for 3D PFC 
simulations with the seed initial condition, 
$\Delta t = 0.01$, $T = 100$. Speedups are reported 
relative to the RTX 2060 baseline. The RTX 2060 
achieves an $8\times$ speedup over a CPU baseline, 
as established in~\cite{Orizaga2024GPU}.}
\label{tab:gpu}
\end{table}

The proposed CS scheme is benchmarked on modern 
NVIDIA Blackwell GPU architectures to demonstrate 
scalability for large-scale 3D PFC simulations. 
Table~\ref{tab:gpu} reports wall-clock times for 
consumer and professional-grade GPUs. Given the 
$8\times$ speedup of the RTX 2060 over a CPU 
baseline~\cite{Orizaga2024GPU}, the RTX 5090 
desktop achieves an effective speedup of 
approximately $70\times$ over CPU, approaching 
the $80\times$ reported for the professional 
NVIDIA A100 in~\cite{Orizaga2024GPU}. The 
professional-grade RTX PRO 6000 reaches $94\times$, 
confirming that the gap between consumer and 
professional hardware is closing for phase-field 
workloads. These results demonstrate that 
large-scale 3D PFC simulations are now accessible 
on consumer hardware without HPC infrastructure.

The RTX 5090 laptop achieves an effective speedup 
of approximately $40\times$ over CPU, comparable 
to the professional A6000 reported 
in~\cite{Orizaga2024GPU}, demonstrating that 
high-resolution 3D PFC simulations are now 
feasible on portable consumer hardware.

\subsection{3D GPU-Accelerated PFC Morphologies}

The proposed scheme naturally captures a rich 
variety of three-dimensional microstructures 
depending solely on the mean-field value of 
the initial condition, without requiring 
prescribed seed configurations. Specifically, 
we consider initial conditions of the form
\[
\phi^0 = 0.01\,\eta + \bar{\phi}, 
\qquad \eta \sim \mathcal{N}(0,1),
\]
where $\bar{\phi}$ denotes the mean density. 
By varying $\bar{\phi}$, three qualitatively 
distinct phases emerge naturally from random 
initial conditions, demonstrating the 
morphological versatility of the convex 
splitting spectral framework. All simulations 
are performed on $256^3$ and $512^3$ grids with 
$r=0.5$, $\Delta t=0.01$, $a=1.0$, and $T=1000$ 
using the consumer-grade RTX 5090 desktop GPU.

\subsubsection{Lamellar Phase ($\bar{\phi}=0$)}

Setting $\bar{\phi} = 0$, the system develops 
a three-dimensional labyrinthine stripe pattern 
distributed throughout the computational domain. 
The fine-scale structure reflects the atomistic 
length scale of the PFC equation, producing 
richly textured microstructure distinct from 
the coarser interfaces observed in classical 
phase-field models. Figure~\ref{fig:3D_stripes} 
shows the isosurface  
at $T=1000$, and Figure~\ref{fig:3D_stripes_energy} 
shows the corresponding free energy.

\begin{figure}[h!]
\centering
\includegraphics[width=0.71\textwidth]{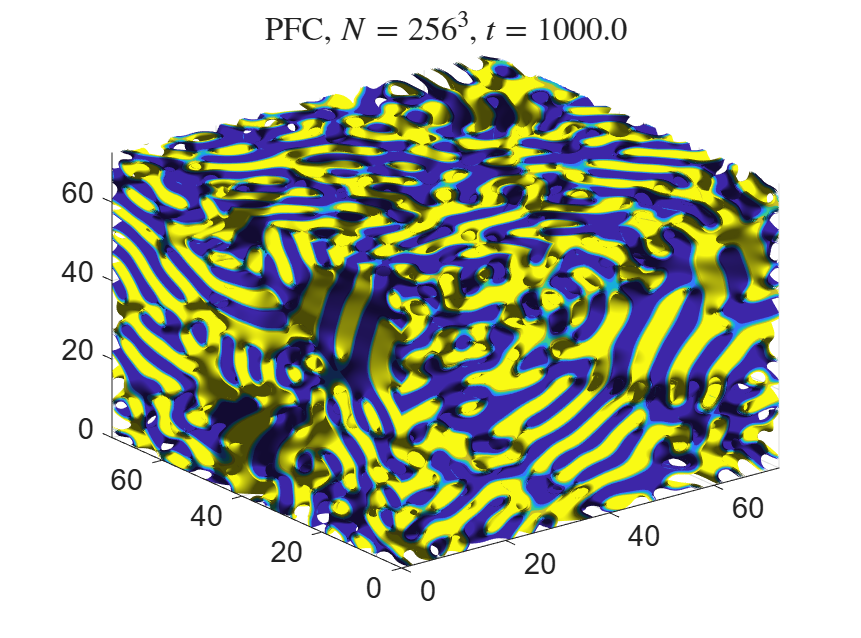} 
\caption{Lamellar phase for the PFC equation with 
$\bar{\phi}=0$, on a $256^3$ grid with $r=0.5$, 
$\Delta t=0.01$, $a=1.0$, $T=1000$.}
\label{fig:3D_stripes}
\end{figure}

\begin{figure}[h!]
\centering
\includegraphics[width=0.7\textwidth]{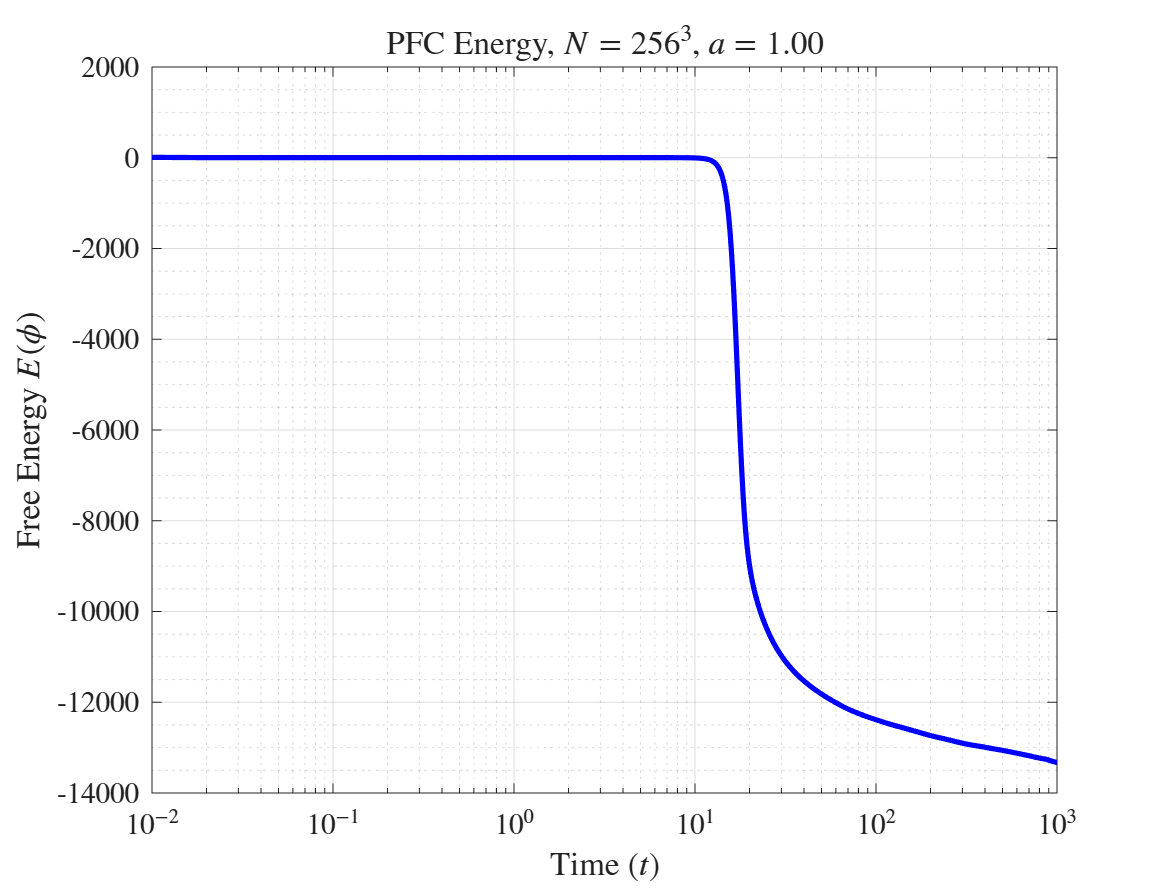}
\caption{Free energy $E(\phi^n)$ as a function of 
time for the lamellar phase simulation of 
Figure~\ref{fig:3D_stripes}, confirming monotonic 
energy dissipation consistent with 
Theorem~\ref{thm:energy}.}
\label{fig:3D_stripes_energy}
\end{figure}

\subsubsection{Spherical Phase ($\bar{\phi}=0.35$)}

Setting $\bar{\phi} = 0.35$, the higher mean 
density favors the formation of isolated spherical 
microdomains distributed throughout the 
computational domain. Unlike the classical 
Cahn--Hilliard equation where spherical domains 
undergo Ostwald ripening, the PFC equation 
stabilizes the spherical phase through its 
atomistic structure, maintaining well-defined 
spherical microdomains throughout the simulation. 
Figure~\ref{fig:3D_spheres} shows the isosurface 
at $T=1000$, and Figure~\ref{fig:3D_spheres_energy} 
shows the corresponding free energy, which decreases 
monotonically consistent with Theorem~\ref{thm:energy}.

\begin{figure}[h!]
\centering
\includegraphics[width=0.71\textwidth]{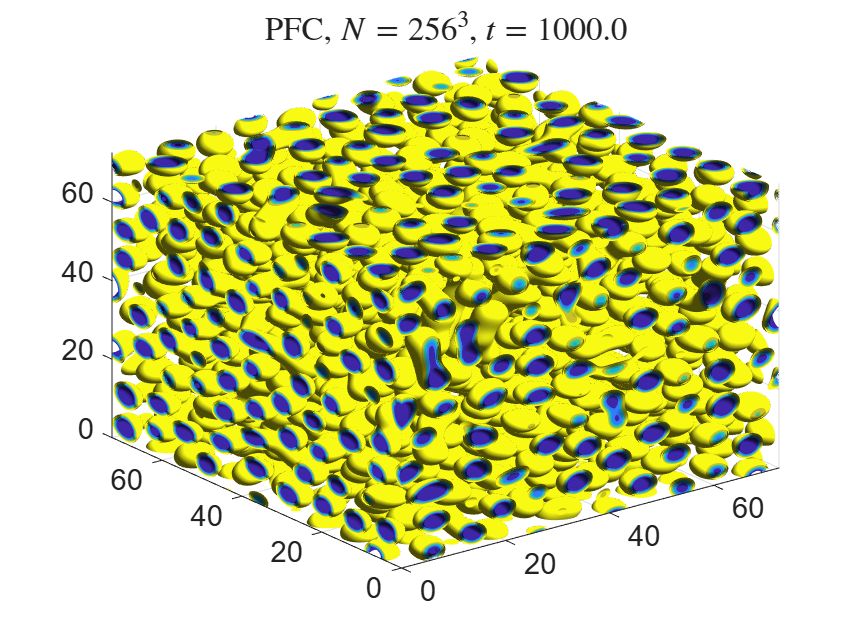}
\caption{Spherical phase for the PFC equation with 
$\bar{\phi}=0.35$, on a $256^3$ grid with $r=0.5$, 
$\Delta t=0.01$, $a=1.0$, $T=1000$.}
\label{fig:3D_spheres}
\end{figure}

\begin{figure}[h!]
\centering
\includegraphics[width=0.7\textwidth]{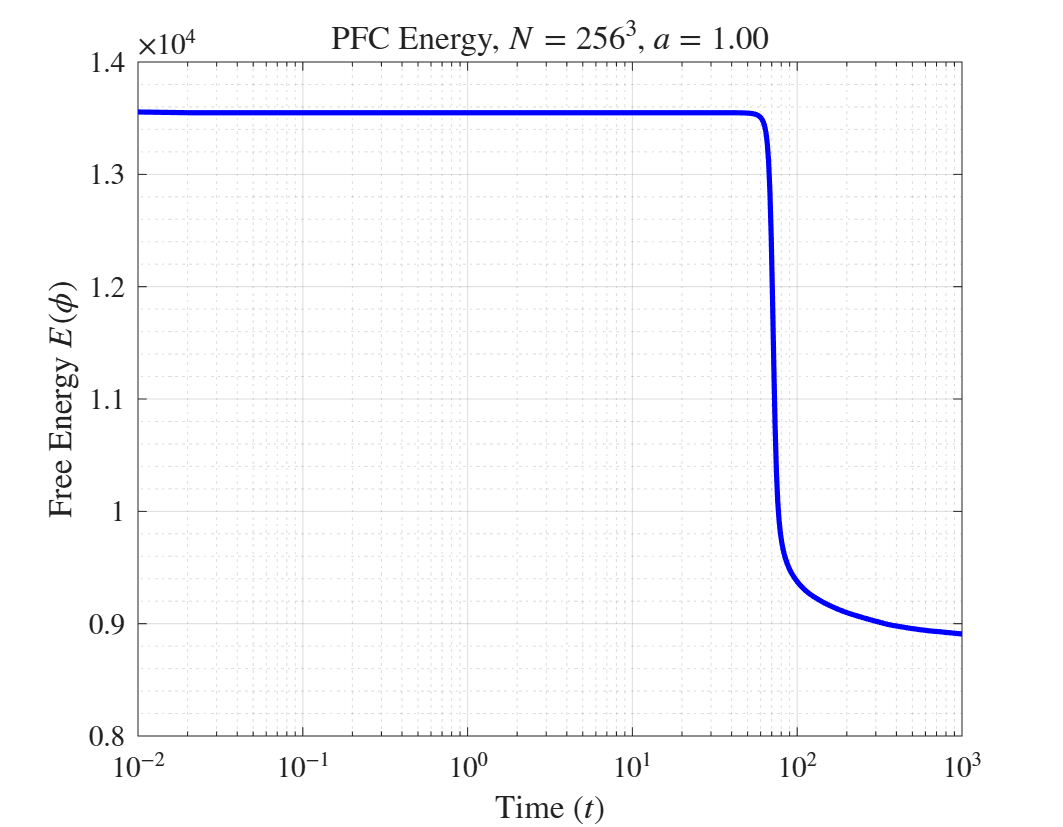}
\caption{Free energy $E(\phi^n)$ as a function of 
time for the spherical phase simulation of 
Figure~\ref{fig:3D_spheres}, confirming monotonic 
energy dissipation consistent with 
Theorem~\ref{thm:energy}.}
\label{fig:3D_spheres_energy}
\end{figure}

\subsubsection{Hexagonal Cylindrical Phase ($\bar{\phi}=0.25$)}

To demonstrate the scalability of the proposed 
framework, we consider an intermediate mean density 
$\bar{\phi} = 0.25$, which produces a hexagonal 
cylindrical phase characterized by interconnected 
tubular microstructures. To fully capture the 
atomistic activity at this scale, we scale up 
to a $512^3$ grid, corresponding to approximately 
134 million degrees of freedom. This resolution 
reveals the rich three-dimensional tubular network 
that emerges from the PFC dynamics, providing 
a stringent test of both the energy stability 
and the GPU scalability of the proposed scheme.
Figure~\ref{fig:3D_hex} shows the isosurface 
and Figure~\ref{fig:3D_hex_energy} shows the 
corresponding free energy.

\begin{figure}[h!]
\centering
\includegraphics[width=0.68\textwidth]
{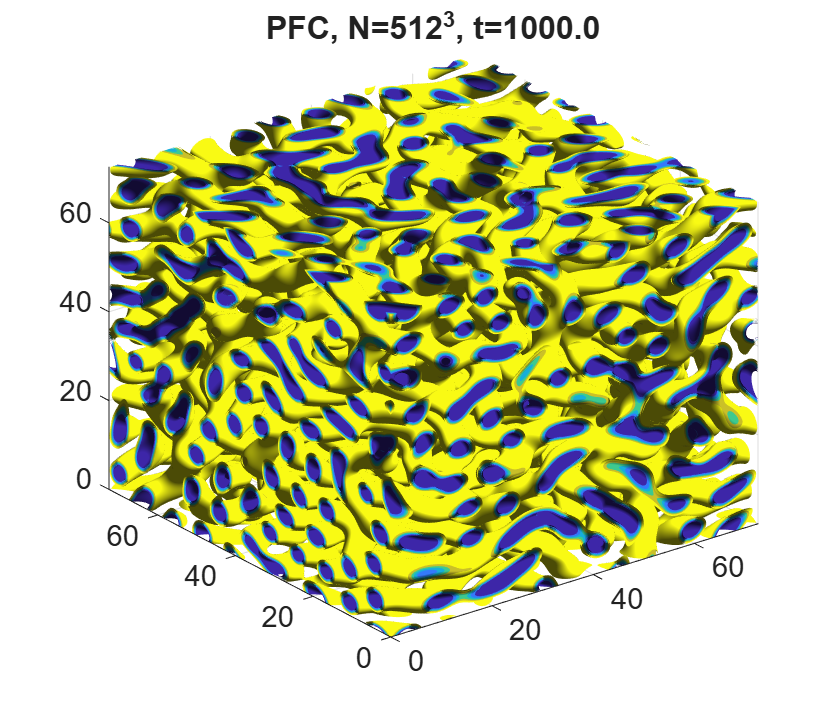}
\caption{Hexagonal cylindrical phase for the PFC 
equation with $\bar{\phi}=0.25$, on a $512^3$ grid 
with $r=0.5$, $\Delta t=0.01$, $a=1.0$, $T=1000$.}
\label{fig:3D_hex}
\end{figure}

\begin{figure}[h!]
\centering
\includegraphics[width=0.7\textwidth]{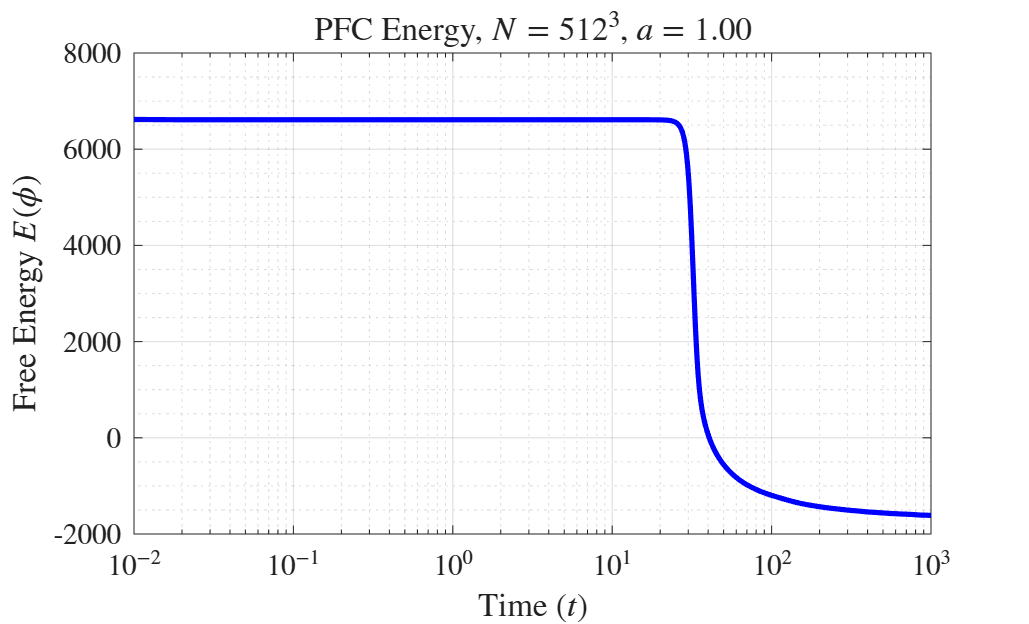}
\caption{Free energy $E(\phi^n)$ as a function of 
time for the hexagonal cylindrical phase simulation 
of Figure~\ref{fig:3D_hex}.}
\label{fig:3D_hex_energy}
\end{figure}


\begin{remark}[GPU Acceleration]
Each $256^3$ simulation presented above runs 
to $T=1000$ with $\Delta t=0.01$, corresponding 
to 100,000 time steps on a grid of $256^3 = 
16{,}777{,}216$ degrees of freedom, over 
16 million unknowns updated at every step. 
On a CPU baseline, such a computation would 
require approximately $94 \times 70 \approx 
6{,}580$ minutes, or over 4.5 days, on a 
single core without HPC infrastructure. 
All of this was completed in under 
\textbf{94 minutes} on a single consumer 
NVIDIA RTX 5090 desktop GPU.
\end{remark}

\begin{remark}[GPU Acceleration at $512^3$]
The $512^3$ hexagonal cylindrical phase simulation 
runs to $T=1000$ with $\Delta t=0.01$, corresponding 
to 100,000 time steps on a grid of $512^3 = 
134{,}217{,}728$ degrees of freedom. On a CPU 
baseline, such a computation would require 
approximately 18 days on a single core without 
HPC infrastructure. On a single consumer GPU, 
this simulation completed in under 6 hours.
\end{remark}

\begin{remark}
Every simulation in this work, including the 40,000 parametric sweeps and the three-dimensional computations, is executed strictly in native double-precision (FP64). As demonstrated in~\cite{Orizaga2024GPU}, single-precision (FP32) introduces an unphysical mass leakage of $\mathcal{O}(10^{-4})$ over long temporal scales. Maintaining full FP64 precision across all GPU configurations guarantees mass conservation to machine precision ($\mathcal{O}(10^{-14})$), ensuring that the reported speedups reflect genuine algorithmic efficiency under uncompromised physical rigor.
\end{remark}

\subsection{Asymptotic Energy Behavior: 
Lamellar Phase at $T=10{,}000$ 
($10^6$ Time Steps)}

To demonstrate the long-time robustness of the 
proposed scheme, we extend the lamellar phase 
simulation to $T=10{,}000$ with $\Delta t=0.01$, 
corresponding to $10^6$ time steps on a $256^3$ 
grid with $\bar{\phi}=0$, $r=0.5$, with $a=1.0$, a value below the classical 
threshold $a \geq 2$ but satisfying the 
derived neutral stability curve~\eqref{eq:neutral}. 
Figure~\ref{fig:longrun_stripes} shows the 
isosurface of the lamellar phase at $T=10{,}000$, 
revealing well-developed three-dimensional lamellar 
stripes. Figure~\ref{fig:longrun_energy} 
shows the corresponding free energy on a log scale, 
demonstrating monotonic dissipation and convergence 
to a well-defined asymptotic plateau, demonstrating monotonic dissipation and convergence toward an asymptotically stable energy plateau. Mass is conserved to machine precision 
throughout all $10^6$ time steps.

\begin{figure}[h!]
\centering
\includegraphics[width=0.7\textwidth]{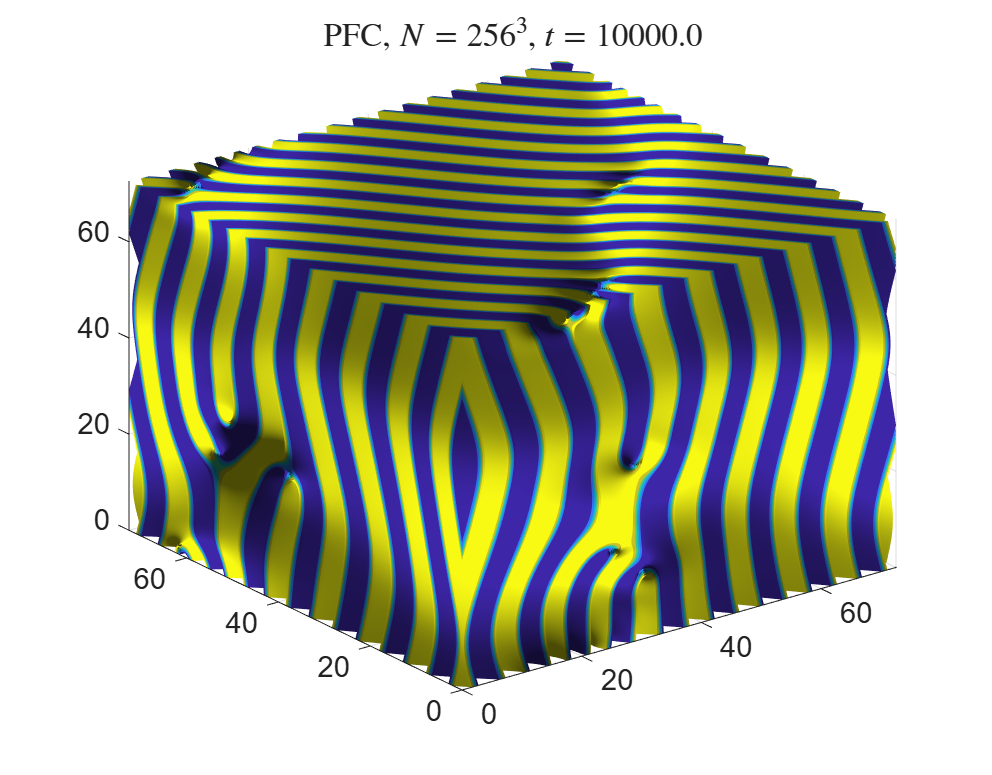}
\caption{Lamellar phase isosurface at $T=10{,}000$ 
for the PFC equation with $\bar{\phi}=0$, on a 
$256^3$ grid with $r=0.5$, $\Delta t=0.01$, 
$a=1.0$.}
\label{fig:longrun_stripes}
\end{figure}

\begin{figure}[h!]
\centering
\includegraphics[width=0.70\textwidth]{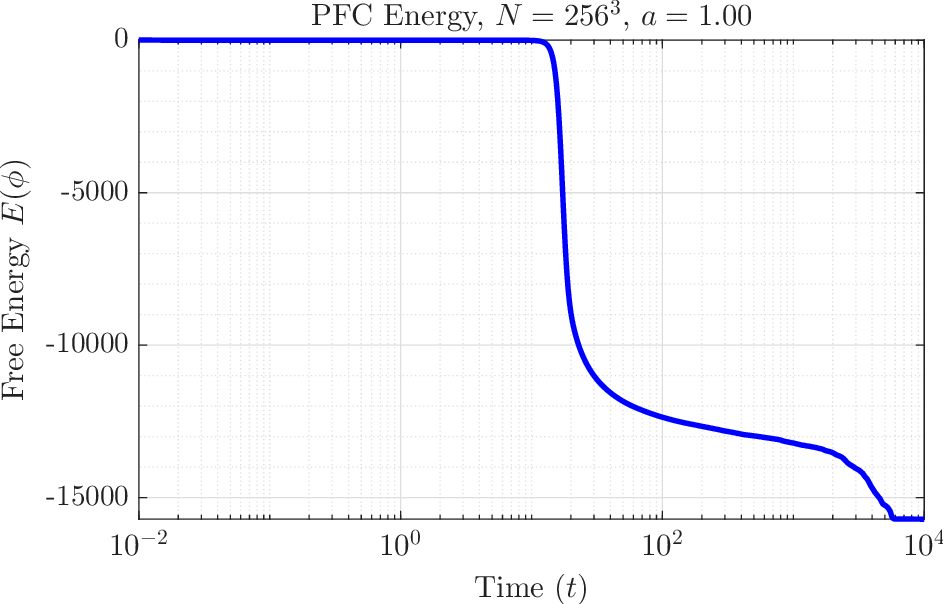}
\caption{Free energy $E(\phi^n)$ as a function of 
time on a log scale up to $T=10{,}000$ 
($10^6$ time steps), demonstrating monotonic dissipation and convergence toward an asymptotically stable energy plateau.}
\label{fig:longrun_energy}
\end{figure}



\section{Conclusions}

We have presented an efficient Fourier spectral method 
based on the convex splitting framework for the phase 
field crystal equation. The proposed first-order scheme 
is unconditionally energy stable, conserves mass to 
machine precision, and admits a closed-form spectral 
update requiring a few lines of MATLAB code.

A central contribution of this work is a computational 
stability analysis of the splitting parameter $a$, 
performed via GPU-accelerated parameter sweeps involving 
40,000 independent PFC simulations. Our results 
demonstrate that the classical sufficient condition 
$a \geq 2$ is not sharp: a semi-analytical neutral 
stability curve derived from a dominant-mode energy 
balance provides a closed-form characterization of 
the practical stability boundary, lying well below 
the classical bound. Moreover, accuracy analysis 
across multiple initial conditions consistently 
demonstrates that smaller values of $a$ within the 
stable region yield lower $L^2$ errors while 
preserving the energy decreasing property, 
providing a practical guideline for parameter selection.

Two-dimensional simulations on $256 \times 256$ 
 and high resolution grids demonstrate the efficiency and morphological 
richness of the proposed framework, including 
crystal growth from a single seed, lamellar and 
hexagonal structures emerging from random initial 
conditions, and grain boundary dynamics from 
polycrystalline configurations. Three-dimensional 
GPU-accelerated simulations at resolutions of 
$128^3$, $256^3$, and $512^3$ on a single consumer 
GPU further demonstrate the scalability of the 
method. Three qualitatively distinct phases emerge 
naturally in three dimensions from random initial 
conditions by varying only the mean density 
$\bar{\phi}$, suggesting that the proposed 
framework is robust for exploring the rich 
morphological landscape of the PFC equation 
in three spatial dimensions. Long-time three-dimensional simulations up to $T=10,000$ ($10^6$ time steps) further demonstrate robustness in the relaxed stability regime predicted by the derived neutral stability curve. These large-scale FP64 computations maintain monotonic energy dissipation and mass conservation to machine precision throughout the entire evolution.

Future work includes the extension of this framework 
to second-order temporal schemes via BDF2--CS and IMEX methods~\cite{OrizagaJCP,gomez2012unconditionally,SongCH}, noting that 
the first-order temporal discretization adopted in the present 
work is a deliberate methodological choice rather than a 
limitation of the underlying convex splitting framework: as 
demonstrated in~\cite{OrizagaJCP}, second-order accuracy via 
BDF2--CS is readily achievable. We note, however, that the 
practical stability threshold for $a$ identified here is specific 
to the first-order temporal discretization analyzed in this work; 
characterizing how this threshold shifts under BDF2--CS or other 
higher-order schemes is a natural and nontrivial direction for 
future investigation, 
adaptive time-stepping strategies guided by the 
neutral stability curve, and applications to 
related phase-field models including the block copolymer 
systems~\cite{OhtaKawasaki1986,bcpOrizaga,Orizagasys}. 
Extensions to thin-film equations with similar 
interfacial dynamics also represent a natural 
direction for future investigation~\cite{TOMtf25,glasnertom,OrizagaTom}. 
Moreover, we plan to extend this framework to 
higher-order phase-field crystal models, including 
the eighth-order PFC equation~\cite{Jaatinen2009,Jaatinen2010}, 
the Lifshitz--Petrich model for quasicrystalline 
structures~\cite{HU2025109439,HU2026129914}, 
which captures additional crystallographic detail 
beyond the standard sixth-order formulation.

The quantitative grain boundary characterization introduced in
Section~\ref{sec:grain_quant} is directly relevant to materials
science applications, and we plan to pursue further work in this
direction, including more detailed defect and grain-growth statistics. Our GPU-based implementation of PFC further paves way for coupling with atomistic simulations, e.g., using \textit{ab initio }and classical molecular dynamics (MD). These MD-based techniques are capable of describing complex structures, e.g., multi-phase interfaces, organic-inorganic molecules, and the like, and their time-evolution through well-established mathematical formulations and interatomic potentials \cite{choudhuri2020interface, choudhuri2022investigation, choudhuri2026untangling, hasan2025metastable,hasan2024transformation}. However, they are limited to probing mechanisms at short time scales, e.g., within pico-nanoseconds, only. On the other hand, PFC simulations can extend to microseconds, and, therefore, can probe physical phenomenon in materials that may be inaccessible to MD. A possible pathway for such multi-temporal coupling is using MD simulations to inform  parameters in the energy functional, e.g. $r$  and $a$ in equations. 3 and 5, which is crucial for PFC simulations. Recent developments in Physics-informed neural networks (PINNs) may allow such information transfer, where PINNs can be trained on MD data while being constrained with a physics-based mathematical function \cite{choudhuri2026untangling,chowdhury2025uncovering}.


The GPU-parallelized parameter sweep strategy
adopted here, motivated by recent advances in
atom-ion dynamics~\cite{chowdhury2026gpu},
is broadly applicable to phase-field and
gradient-flow models requiring systematic
stability or sensitivity analysis across
large parameter spaces. Many modern numerical 
frameworks, including higher-order convex 
splitting, SAV, and IMEX-type schemes, introduce 
stabilization or splitting parameters whose 
practical stability regions invite further 
computational exploration. The present
GPU-parallel stability-map methodology provides
a scalable and reproducible framework for
systematically exploring such parameter landscapes
through large-scale PDE simulations on consumer
hardware.

\section*{Acknowledgments}
This research was supported by an Institutional 
Development Award (IDeA) from the National Institute 
of General Medical Sciences of the National Institutes 
of Health under grant number P20GM103451. S.O 
acknowledges the University of New Mexico (UNM) for 
providing MATLAB licensing and computational resources 
used in this work. S.O.\ thanks Maurice Fabien from 
MIT for the use of a professional GPU. P. Yin's 
research was supported by the University of Texas at 
El Paso Startup Award. D. Choudhuri's research was 
supported by the National Science Foundation under 
Award No. 2333630.

\section*{Declaration of Competing Interests}
The authors declare that they have no known competing 
financial interests or personal relationships that 
could have appeared to influence the work reported 
in this paper.

\section*{Data and Code Availability}

In the interest of full scientific 
reproducibility, all source scripts are hosted on a dedicated GitHub repository 
at \url{https://github.com/sauloorizaga/CS-PFC} and will be made publicly 
accessible upon the formal acceptance of this manuscript.

\end{document}